\documentclass[a4paper,12pt]{article}
\usepackage{graphicx,amssymb,amstext,amsmath}

\usepackage[latin1]{inputenc}
\usepackage{amsthm}
\usepackage{dsfont}
\usepackage{amsmath}
\usepackage{amsfonts}
\usepackage{amssymb}
\usepackage{mathrsfs}
\usepackage{graphicx}
\usepackage{cancel}
\usepackage{comment}
\usepackage{calligra}
\usepackage{graphicx,color}
\usepackage{tikz}
\usetikzlibrary{arrows.meta, positioning}
\usepackage[all]{xy}
\usepackage[total={5.2in,9.1in},
top=1.4in, left=1.55in, includefoot]{geometry}

\usepackage{amsmath}
\usepackage{amssymb}
\usepackage{amsthm}
\newtheorem{teo}{Theorem}[section]
\newtheorem{lem}[teo]{Lemma}

\newtheorem{coro}[teo]{Corollary}
\newtheorem{defi}[teo]{Definition}

\newtheorem{remark}[teo]{Remark}

\newcommand{\re}{\mathbb{R}}
\newcommand{\ds}{\displaystyle}

\begin{document}
\thispagestyle{empty}

\title{\textbf{\Large{On the Birkhoff Spectrum for Hyperbolic Dynamics}}}
\author{ Sergio Roma\~na}
\date{}
\maketitle
\begin{abstract}
In this paper, we study the structure of Birkhoff spectra for hyperbolic dynamical systems. Given a H\"older observable \(f\) on a basic set \(\Lambda\), we obtain the following results: First, we characterize when the Birkhoff spectrum of \(f\) is dense in the positive (or negative) real line. Second, we prove that a bounded Birkhoff spectrum forces \(f\) to be cohomologous to zero, which constitutes an extension of Liv\v{s}ic's theorem. Moreover, we show that if the spectrum exhibits an ``arithmetically sparse'' structure, then \(f\) is cohomologous to a constant. \\
\indent We then extend these results to continuous time. For Anosov flows---including geodesic flows on Anosov manifolds---we establish analogous density results for Birkhoff integrals over closed orbits. In particular, we generalize a theorem of Dairbekov--Sharafutdinov \cite{Dairbekov} by proving that a bounded (resp.~arithmetically sparse) spectrum forces a smooth function to vanish (resp.~be constant).

\begin{quote}
\end{quote}
\end{abstract}

\section{Introduction}
The study of asymptotic invariants associated with dynamical systems is central to understanding their statistical and geometric properties. A classical invariant is the set of \emph{Birkhoff sums} along periodic orbits. Given a discrete dynamical system \(\varphi \colon M \to M\) and a continuous function \(f \colon M \to \mathbb{R}\), the \emph{Birkhoff spectrum} is defined as
\[
\mathcal{B}(f,\varphi)=\Big\{S(f, \varphi, p): p\in \operatorname{Per}(\varphi)\Big\},
\]
where 
\[
S(f, \varphi, p):=\sum_{i=0}^{\tau_p-1}f(\varphi^i(p)),\quad
\operatorname{Per}(\varphi)=\text{set of periodic points},
\]
and \(\tau_p\) is the period of \(p\).  

When the dynamics are restricted to an invariant subset \(\Lambda\), the corresponding set is denoted by \(\mathcal{B}(f,\varphi, \Lambda)\).

These objects have a long history in ergodic theory and differentiable dynamics. In particular, the celebrated Liv\v{s}ic theorem \cite{Livsic} characterizes H\"older functions with vanishing Birkhoff spectrum on hyperbolic sets as those cohomologous to zero. More recently, the structure of \(\mathcal{B}(f,\varphi)\)---especially its density and distribution properties---has been studied for Anosov systems and basic sets \cite{Dairbekov,Shaobo,Sigmund2}. Such questions are deeply linked to the distribution of periodic orbits and the flexibility of invariant measures supported on them.

This work, inspired by \cite{Shaobo}, explores the density of Birkhoff spectra for some H\"older continuous observables, establishing new rigidity results for these spectra on basic sets of diffeomorphisms and flows.


\subsection*{Main Results for diffeomorphisms}
A celebrated result, Liv\v{s}ic's Theorem \cite{Livsic}, characterizes when a H\"older continuous function \(f\) has a trivial Birkhoff spectrum on a basic set, i.e., when \(\mathcal{B}(f, \varphi, \Lambda) = 0\). It establishes that this occurs if and only if \(f\) is cohomologous to a constant---specifically, if there exists a H\"older continuous function \(u\) such that \(f = u \circ \varphi - u\). In this case, we say that \(f\) is \emph{cohomologous to zero}.\\
\indent Our first result establishes a sharp rigidity principle for H\"older observables
on basic sets: the boundedness of the Birkhoff spectra \(\mathcal{B}(f,\varphi,\Lambda)\)
forces the observable to be cohomologous to a constant. More specifically,
\begin{teo}\label{mainT}
Let \(\Lambda\subset U\) be a basic set for \(\varphi \colon U \to M\) and let \(f \colon U \to \mathbb{R}\) be H\"older continuous. The following are equivalent:
\begin{itemize}
\item[\emph{(i)}] \(\mathcal{B}(f,\varphi, \Lambda)\) is bounded;
\item[\emph{(ii)}] \(f\) is cohomologous to zero;
\item[\emph{(iii)}] \(\mathcal{B}(f,\varphi, \Lambda)=\{0\}\).
\end{itemize}
\end{teo}
This equivalence turns a mild, qualitative condition---boundedness---into a powerful geometric constraint, 
showing that the only way the set of Birkhoff integrals can avoid being unbounded 
is if all invariant measures yield the same average. In particular, it extends 
Liv\v{s}ic's theorem from trivial spectrum case to arbitrary boundedness condition and reveals 
a fundamental dichotomy: either \(\mathcal{B}(f,\varphi,\Lambda)\) is a singleton, 
or it is unbounded. This result lies at the heart of the geometric theory of hyperbolic dynamics, with direct implications for thermodynamic formalism, and the multifractal analysis of Birkhoff spectra.

Using this result we obtain all the solutions obtain in \cite[Theorem 1.1]{ZouWei} for the finite approximation of Liv\v sic Theorem are trivial (see Section \ref{Proof mainT}).\\
\ \\
To state our second result, we introduce a classification of observables based on their spectral support. A continuous observable \(f\) is termed \emph{dispersed} if its Birkhoff spectrum intersects both the positive and negative reals, i.e.,
\[
\mathcal{B}(f,\varphi, \Lambda)\cap \mathbb{R}^+ \neq \emptyset \quad \text{and} \quad \mathcal{B}(f,\varphi, \Lambda)\cap \mathbb{R}^- \neq \emptyset.
\]
Conversely, if \(\mathcal{B}(f,\varphi, \Lambda) \subset \mathbb{R}_{\geq 0}\) or \(\mathcal{B}(f,\varphi, \Lambda) \subset \mathbb{R}_{\leq 0}\), we call \(f\) \emph{concentrated}.\\
Moreover, a continuous observable \(f\) is called \emph{arithmetic} if there is $c\in\re$ with 
$$\mathcal{B}(f,\varphi, \Lambda) \subset c\mathbb{Z}.$$
In \cite{Shaobo}, it was shown that for transitive Anosov diffeomorphisms and any dispersed observable, the Birkhoff spectrum is dense in \(\mathbb{R}\). In the same paper, the density result was obtained for basic sets (instead of Anosov systems) for dispersed and non-arithmetic observable (see Theorem \ref{Shaobo} in Section \ref{New Shaobo} for details). However, in Section \ref{example}, we provide an example of a basic set and a dispersed non-arithmetic observable for which density of the Birkhoff spectra \emph{fails}. This suggests that a sufficient condition for density likely involves additional topological restrictions on the basic set. The techniques developed in the proofs of Theorem \ref{Shaobo} and Theorem \ref{mainT} can also be used to study the \emph{average Birkhoff spectrum} associated with a H\"older observable (see Section \ref{measure spectra} for details).\\

Concentrated observables preclude spectral density in $\re$, our subsequent result provides sufficient conditions to achieve density in the maximal admissible subset.
\begin{teo}\label{mainT2}
Let \(\Lambda\subset U\) be a basic set for \(\varphi \colon U \to M\) and let \(f \colon U \to \mathbb{R}\) be a  H\"older continuous concentrated function. Then, if 
    \(f\) is not cohomologous to zero and \(\inf \mathcal{B}(f,\varphi, \Lambda) = 0\) \emph{(}resp. $\sup \mathcal{B}(f,\varphi, \Lambda) = 0$\emph{)} if and only if \(\mathcal{B}(f,\varphi, \Lambda)\) is dense in \([0,\infty)\) \emph{(}resp.  \((-\infty,0]\)\emph{)}. 
\end{teo}
This result can be interpreted as a \emph{concentrated induced denseness} phenomenon: positivity (or negativity) of the observable, together with the infimum (or maximum)being zero, forces the range to be as large as possible$-$that is, dense in \([0,\infty)\) (or $(-\infty, 0]$).\\

While Theorem \ref{mainT} shows that \emph{boundedness} of \(\mathcal{B}(f,\varphi,\Lambda)\) implies
that \(f\) is cohomologous to zero, the following result demonstrates that a significantly
weaker restriction$-$namely, the set being contained in a finite union of dilated copies of
\(\mathbb{Z}-\)forces the similar conclusion. This can be interpreted as a \emph{quantization rigidity} 
property: the only way for the Birkhoff spectra to be ``arithmetically sparse'' is 
if all Birkhoff averages coincide.

\begin{teo}\label{Rigidity}
Let \(\Lambda\subset U\) be a basic set for \(\varphi \colon U \to M\) and let \(f\) be a H\"older continuous non-arithmetic observable. If 
\[
\mathcal{B}(f,\varphi, \Lambda) \subset a_1\mathbb{Z} \cup \dots \cup a_l\mathbb{Z}
\quad\text{for some } a_1,\dots,a_l \in \mathbb{R}^{+},
\]
then \(f\) is cohomologous to a constant \(a_{i_0}\) for some \(i_0\).
\end{teo}

The non-arithmeticity condition on \(f\) is necessary, as shown in Section \ref{example}.


\subsection*{Extensions to flows}
The last part of the paper deals with continuous-time systems. For a flow \(\phi^t\), a closed orbit \(\mathcal{O}_\theta\) of period \(\tau_\theta\), and a parametrization $\mathcal{O}_\theta(s)$ define the integral
\[
\oint_{\mathcal{O}_\theta} f = \int_{0}^{\tau_\theta} f(\mathcal{O}_\theta(s))\,ds,
\]
and the Birkhoff spectrum
\[
\mathcal{B}(f,\phi^t)=\Big\{\oint_{\mathcal{O}_{\theta}} f : \theta \in \operatorname{Per}(\phi^t)\Big\}.
\]
We obtain the following analogue of Theorem \ref{Shaobo}.
\begin{teo}\label{MainT2}
Let \(\phi^t \colon U \to N\) be a flow, \(\Lambda\subset U\) a basic set for \(\phi^t\), and \(f \colon U \to \mathbb{R}\) a dispersed and non-arithmetic H\"older continuous function. Then \(\mathcal{B}(f,\phi^t)\) is dense in \(\mathbb{R}\). In particular, for Anosov flows (including transitive ones) the conclusion holds without requiring the non-arithmeticity condition.
\end{teo}

When $f$ is not dispersed, results analogous to Theorems \ref{mainT}, \ref{mainT2}, and \ref{Rigidity}, are obtained in the setting of Anosov flows.\\

\ \\
An important class of Anosov flows arises in geometry: the geodesic flow of an \emph{Anosov manifold} \((M,g)\)$-$ a compact Riemannian manifold whose geodesic flow is Anosov. This includes all manifolds of negative sectional curvature, but also certain manifolds with mixed curvature \cite{Anosov, CR2, Eberlein, IR2, IR1}. 
A Riemannian metric \( g \) on \( M \) is called an \emph{Anosov metric} if its geodesic flow \( \phi^t_M:SM\to SM \) is Anosov; in this case, the pair \( (M,g) \) is referred to as an \emph{Anosov manifold}. Recall that $SM$ is the unitary tangent bundle.\\
In contrast with Theorems \ref{mainT} and \ref{mainT2} for the flow case, the geodesic flow exhibits a fundamentally different behavior. As shown in Section \ref{geodesic flow} (see Remark \ref{No bounded}), for the specific choice $f=\text{Ric}$ (the Ricci curvature), the associated   \(\mathcal{B}(\mathrm{Ric},\phi^t)\) is unbounded, and $0$ is not an accumulation point of its spectrum.

\noindent Since the geodesic flow preserves the Liouville measure, it is transitive (see \cite{Pa}). Therefore, as a direct application of Theorem~\ref{MainT2}, we obtain:
\begin{coro}\label{Anosov geodesic flow}
Let \((M,g)\) be an Anosov manifold with geodesic flow \(\phi^t_M\). For every dispersed and non-arithmetic H\"older function \(f\),
\[
\overline{\mathcal{B}(f,\phi^t_M)} = \mathbb{R}.
\]
\end{coro}

The Birkhoff spectrum also provides geometric rigidity when applied to curvature observables. Specifically, when the observable is the Ricci curvature \(\mathrm{Ric}\), we obtain some characterizations (see Section \ref{geodesic flow}). \\

\noindent If we restrict to functions on the manifold itself, we define
\[
\mathcal{B}_M(f) := \Big\{\oint_{\gamma_\theta} f : \gamma_\theta \text{ is a closed geodesic in } M\Big\}.
\]

For this spectrum, we obtain a generalization of a result of Dairbekov and Sharafutdinov \cite{Dairbekov} in two point of view:

\begin{teo}\label{integral_geodesic}
Let \((M,g)\) be an Anosov manifold and \(f \in C^\infty(M)\). 
\begin{itemize}
    \item[\text{\emph{(a)}}] If \(\mathcal{B}_M(f)\) is bounded, then \(f\) vanishes identically.
    \item[\text{\emph{(b)}}]  If $f$ is non-arithmetic and \(\mathcal{B}(f,\varphi^t) \subset a_1\mathbb{Z} \cup \dots \cup a_l\mathbb{Z}\) for some \(a_1,\dots,a_l \in \mathbb{R}\). Then \(f\) is  constant \(a_{i_0}\) for some \(i_0\).
\end{itemize}
\end{teo}
While condition (b) is analogous to that of Theorem \ref{Rigidity}, it leads to a strictly stronger result: here the function must be constant, rather than just cohomologous to a constant.\\
This result reveals a strong rigidity property in Anosov geometry: 
boundedness of the closed geodesic spectrum forces the observable to vanish completely. 
Not only does it generalize the theorem of Dairbekov and Sharafutdinov (see \cite[Theorem 1.1]{Dairbekov}) from the 
special case \(\mathcal{B}_M(f)=\{0\}\) to arbitrary bounded spectra, but it also 
illustrates a fundamental principle in the study of Anosov systems: any constraint 
on the distribution of closed geodesic integrals imposes drastic restrictions on 
the underlying function.

The theorem provides a powerful tool for studying geometric inverse problems on 
Anosov manifolds---such as spectral rigidity, lens rigidity, and the marked length 
spectrum problem---where the behavior of integrals along closed geodesics plays a 
central role. It shows that boundedness of \(\mathcal{B}_M(f)\) is as rigid as the 
condition of being identically zero on all closed geodesics, emphasizing the 
profound interplay between the dynamics of the geodesic flow and the geometry of 
the manifold.

In essence, this sharpens our understanding of how the geodesic flow controls 
smooth functions: if one observes that the set of integrals of a smooth function 
over closed geodesics is merely bounded, one must conclude the function vanishes 
everywhere. This is a striking geometric uniqueness statement with clear 
dynamical significance.

\subsection*{Organization of the paper}
The paper is organized as follows. Section~\ref{sec:prelim} recalls necessary background on hyperbolic dynamics, periodic measures, and the Liv\v sic Theorem. Section~\ref{Section for Diffeos} contains the proofs of our main theorems for diffeomorphisms, including an application to the average Birkhoff spectrum (Theorem~\ref{main2}). Section~\ref{Section for Flows} extends these results to flows. After reducing the problem via a Poincar'e map in Section~\ref{reduction via Poincare Maps}, we prove Theorems~\ref{MainT2} and~\ref{integral_geodesic}. We conclude with applications to geodesic flows in Section~\ref{geodesic flow}.

\section{Preliminaries}\label{sec:prelim}

\subsection{Hyperbolic sets and basic sets}
Let \(M\) be a smooth Riemannian manifold and \(\varphi: M \to M\) a \(C^1\) diffeomorphism. A compact \(\varphi\)-invariant set \(\Lambda \subset M\) is called \emph{hyperbolic} if there exists a continuous \(D\varphi\)-invariant splitting \(T_\Lambda M = E^s \oplus E^u\) and constants \(C > 0\), \(0 < \lambda < 1\) such that for all \(x \in \Lambda\) and \(n \geq 0\),
\[
\|D\varphi^n(v)\| \leq C\lambda^n\|v\|\ \text{for } v \in E^s_x, \quad
\|D\varphi^{-n}(v)\| \leq C\lambda^n\|v\|\ \text{for } v \in E^u_x.
\]
A hyperbolic set \(\Lambda\) is said to be a \emph{basic set} if
\begin{itemize}
    \item \(\varphi|_\Lambda\) is transitive (there exists a dense orbit),
    \item periodic points are dense in \(\Lambda\),
    \item \(\Lambda\) is locally maximal: there exists an open neighbourhood \(V\) of \(\Lambda\) such that \(\Lambda = \bigcap_{n \in \mathbb{Z}} \varphi^n(\overline{V})\).
\end{itemize}
A diffeomorphism \(\varphi\) satisfies \emph{Axiom A} if its non-wandering set \(\Omega(\varphi)\) is hyperbolic and the periodic points are dense in \(\Omega(\varphi)\). Smale's Spectral Decomposition Theorem states that for an Axiom A diffeomorphism, \(\Omega(\varphi)\) splits into finitely many pairwise disjoint basic sets:
\[
\Omega(\varphi) = \Omega_1 \cup \cdots \cup \Omega_k.
\]
Using this decomposition, we have the following definition: 
\begin{defi}
    A continuous real-valued function $f$ on $\Omega$ is called \emph{dispersed} if there exists an index $i_0 \in {1, \dots, k}$ such that:
$\mathcal{B}(f,\varphi, \Omega_{{i_0}})\cap \mathbb{R}^{\pm}\neq \emptyset$. Otherwise, the function $f$ is said to be \emph{concentrated}. \\
Moreover, a continuous observable \(f\) is called \emph{arithmetic} if for all index $i \in {1, \dots, k}$ there is $c_i\in\re$ with 
$$\mathcal{B}(f,\varphi, \Omega_{{i}}) \subset c_i\mathbb{Z}.$$
\end{defi} 

The density in the theorems of the last section will be achieved only in the dispersed and non-arithmetic function of H\"older. 

All these definitions and properties apply to flows.

\subsection{Some Hyperbolic Properties} 
\subsubsection{Local Product Structure}
The stable bundle \( E^s \) and unstable bundle \( E^u \) are uniquely integrable, that is,  for each \( x \in \Lambda \), there exists a unique \( C^1 \) immersed submanifold \( W^s(x) \) (resp. \( W^u(x) \)) such that \( T_y W^s(x) = E^s(y) \) for all \( y \in W^s(x) \) (and similarly for \( W^u(x) \)). Locally,  for \( x \in \Lambda \) and sufficiently small \( \delta > 0 \): (sometimes denotes by $W^{s}_{loc}(x)$)
\[
W^s_\delta(x) = \left\{ y \in M : d(f^n(y), f^n(x)) \leq \delta \text{ for all } n \geq 0 \right\}.
\]
In words: the set of points whose forward orbits remain within \( \delta \)$-$distance of the orbit of \( x \). This is a 
\( C^1 \) embedded disk tangent to \( E^s(x) \) at \( x \), satisfying:
\begin{itemize}
    \item \textbf{Forward invariance:} \( f(W^s_\delta(x)) \subset W^s_\delta(f(x)) \)
    \item \textbf{Exponential contraction:} \( d(f^n(y), f^n(x)) \leq C\lambda^n d(y,x) \) for some \( C>0 \), \( \lambda \in (0,1) \)
\end{itemize}
The global stable manifold is obtained by 
\[
W^s(x) = \bigcup_{n \geq 0} f^{-n}(W^s_\delta(f^n(x))) = \left\{ y \in M : \lim_{n\to\infty} d(f^n(y), f^n(x)) = 0 \right\}.
\]
In words: all points whose forward orbits asymptotically converge to the orbit of \( x \). Similar definitions hold for the unstable case.

\begin{defi}[Bracket Operation]
Let $\Lambda$ be a hyperbolic basic set for a diffeomorphism $f: M \to M$, and let $\epsilon > 0$ be sufficiently small such that the local stable and unstable manifolds $W^s_{loc}(x)$ and $W^u_{loc}(x)$ are well-defined for all $x \in \Lambda$.\\
The \emph{bracket operation} is defined on pairs of sufficiently close points in $\Lambda$:
\[
[\cdot,\cdot]: \{(x,y) \in \Lambda \times \Lambda \mid d(x,y) < \delta\} \to M
\]
where $\delta > 0$ is chosen small enough to ensure unique intersections. The operation is given by:
\[
[x,y] := W^u_{loc}(x) \pitchfork W^{s}_{loc}(y)
\]
That is, $[x,y]$ is the unique intersection point of the local unstable manifold of $x$ and the local stable manifold of $y$.
\end{defi}

\begin{remark}\label{Homoclinic}
    For a basic set $\Lambda$, any pair of $p,q\in \text{Per}(\varphi)\cap \Lambda$ are homoclinic related, that is,  
    $$W^u(p) \pitchfork W^{s}(q) \,\,\text{and}\,\, W^u(q) \pitchfork W^{s}(p).$$
\end{remark}

\begin{lem}\label{Bracket Lemma}

For every $\varepsilon > 0$, there exists $\delta > 0$ such that if $d(x,y) < \delta$ for $x,y \in \Lambda$, then:
\[
\max\left\{d(x, [x,y]), d(x, [y,x]), d(y, [x,y]), d(y, [y,x])\right\} < \varepsilon
\]
\end{lem}

\begin{proof}
    
\end{proof}

\begin{proof}
We will show that $d(x, [x,y]) < \varepsilon$; the other three bounds follow by identical reasoning.
The bracket operation is defined as: \[[x,y] = W^u_\varepsilon(x) \cap W^s_\varepsilon(y)\]
where $W^u_\varepsilon(x)$ and $W^s_\varepsilon(y)$ are the local unstable and stable manifolds at $x$ and $y$ respectively.

The key observations are:
\begin{itemize}
    \item The local stable and unstable manifolds vary continuously with the base point.
    \item The unique intersection $W^u_\varepsilon(x) \cap W^s_\varepsilon(x) = \{x\}$.
    \item Since the stable and unstable manifolds always intersect transversely (at a non-zero angle), a small perturbation from $W^s_\varepsilon(x)$ to $W^s_\varepsilon(y)$ causes the intersection point to move only slightly along $W^u_\varepsilon(x)$.
\end{itemize}
More formally: the map $(x,y) \mapsto [x,y]$ is continuous (this follows from the continuous dependence of invariant manifolds and transverse intersection theory). Since $[x,x] = x$, by the definition of continuity, there exists $\delta > 0$ such that if $d(x,y) < \delta$, then $d(x, [x,y]) < \varepsilon$.

The same argument applies to $[y,x] = W^u_\varepsilon(y) \cap W^s_\varepsilon(x)$, giving the other three bounds by symmetry.
\end{proof}
Observe that for any fixed $\theta \in (0,1)$, the function $t \mapsto t^\theta$ is continuous at $0$ with $0^\theta = 0$. Then we have 
\begin{coro}\label{Coro- Bracket}
    Given any $0 < \theta < 1$, for every $\varepsilon > 0$ there exists $\delta > 0$ such that if $d(x,y) < \delta$ for $x,y \in \Lambda$, then:
\[
\max\left\{d(x, [x,y])^\theta,\ d(x, [y,x])^\theta,\ d(y, [x,y])^\theta,\ d(y, [y,x])^\theta\right\} < \varepsilon
.\]
\end{coro}
\begin{remark}
    Stable and unstable manifolds also exist for flows. In this setting, one considers 
the \emph{center-stable} and \emph{center-unstable} manifolds, defined respectively by
\[
W^{cs}(x) = \bigcup_{t \in \mathbb{R}} W^s(\phi^t(x))
\quad \text{and} \quad
W^{cu}(x) = \bigcup_{t \in \mathbb{R}} W^u(\phi^t(x)).
\]
These are invariant families that contain the full flow lines through the local stable 
and unstable leaves.
\end{remark}
\subsubsection{Shadowing Property } 
Given \(\delta > 0\), a sequence \((x_n)_{n \in \mathbb{Z}} \subset M\) is called a \emph{\(\delta\)-pseudo-orbit} for \(\varphi\) if
\[
d(\varphi(x_n), x_{n+1}) \leq \delta \quad \text{for all } n \in \mathbb{Z}.
\]
If there exists \(p > 0\) such that \(x_{n+p} = x_n\) for all \(n\), the pseudo-orbit is called \emph{periodic} (or a \emph{periodic \(\delta\)-pseudo-orbit}).\\
We say that a point \(y \in M\) \emph{\(\varepsilon\)-shadows} the pseudo-orbit \((x_n)\) if
\[
d(\varphi^n(y), x_n) \leq \varepsilon \quad \text{for all } n \in \mathbb{Z}.
\]
\begin{lem}\emph{\cite[Theorem 1]{Sakai}}
    For a hyperbolic set \(\Lambda\), there exist \(\delta_0 > 0\) and $\mu\geq 1$ such that every \(\delta\)-pseudo-orbit in \(\Lambda\) with \(\delta \leq \delta_0\) is \(\mu\delta \)-shadowed by a true orbit in \(\Lambda\). Moreover, if the pseudo orbit is periodic, then the shadow is a periodic orbit. 
\end{lem}

\subsection{Invariant measures and periodic measures}
Let \(\mathcal{M}(\Lambda,\varphi)\) denote the space of \(\varphi\)-invariant Borel probability measures supported on \(\Lambda\). This is a convex compact metrizable space in the weak-* topology. The ergodic measures form the extremal points of \(\mathcal{M}(\Lambda,\varphi)\).\\
For each periodic point \(p \in \operatorname{Per}(\varphi)\) of period \(\tau_p\), we associate the periodic measure
\[
\mu_p := \frac{1}{\tau_p}\sum_{i=0}^{\tau_p-1} \delta_{\varphi^i(p)}.
\]
A fundamental fact in hyperbolic dynamics is that for a basic set \(\Lambda\), the set \(\{\mu_p : p \in \operatorname{Per}(\varphi) \cap \Lambda\}\) is dense in \(\mathcal{M}(\Lambda,\varphi)\) (see \cite{Sigmund1}).

\subsection{Liv\v{s}ic Theorem and cohomology}
Two continuous functions \(f, g: \Lambda \to \mathbb{R}\) are said to be \emph{cohomologous} (with respect to \(\varphi\)) if there exists a continuous function \(u: \Lambda \to \mathbb{R}\) such that
\[
f - g = u \circ \varphi - u.
\]
For H\"older continuous functions on a hyperbolic set, the celebrated Liv\v{s}ic theorem \cite{Livsic} states that \(f\) is cohomologous to zero if and only if \(S(f,\varphi,p) = 0\) for every periodic point \(p \in \Lambda\). More generally, if \(f\) and \(g\) have the same Birkhoff sums at every periodic point, then they are cohomologous.

This rigidity plays a central role in our analysis of concentrated functions (Theorems~\ref{mainT} and~\ref{Rigidity}).

\subsection{Functions on geodesic flows}
Let \((M,g)\) be a closed Riemannian manifold and \(SM\) its unit tangent bundle. The geodesic flow \(\phi^t: SM \to SM\) is defined by \(\phi^t(x,v) = (\gamma_{x,v}(t), \dot{\gamma}_{x,v}(t))\), where \(\gamma_{x,v}\) is the geodesic with initial conditions \((x,v)\). A closed geodesic of period \(\tau\) corresponds to a periodic orbit of \(\phi^t\).

A Riemannian metric is called \emph{Anosov} if its geodesic flow is an Anosov flow on \(SM\). In this case, \((M,g)\) is called an \emph{Anosov manifold}. Important examples include manifolds of negative sectional curvature, but the class is strictly larger \cite{Eberlein, IR2, IR1}.

For a function \(f: M \to \mathbb{R}\), its integral along a closed geodesic \(\gamma\) of length \(\ell(\gamma)\) is
\[
\oint_\gamma f = \int_0^{\ell(\gamma)} f(\gamma(s))\,ds,
\]
which coincides with the Birkhoff integral for the flow when \(f\) is lifted to \(SM\) as \(f(x,v) := f(x)\).

\section{Proof of Main Results for Diffeomorphisms}\label{Section for Diffeos}

\subsection{Density of the Birkhoff Spectra}\label{New Shaobo}
In this section, we establish conditions that guarantee density of the spectra for basic sets. This result was obtained in \cite{Shaobo}; however, our exposition aims to clarify the origin of the "non-arithmeticity" condition. More precisely,
\begin{teo}[\cite{Shaobo}]\label{Shaobo}
Let \(\Lambda \subset U\) be a basic set for \(\varphi \colon U \to M\) and let \(f \colon U \to \mathbb{R}\) be a H\"older continuous function. If $f$ is dispersed and non-arithmetic, then \(\mathcal{B}(f,\varphi, \Lambda)\) is dense in \(\mathbb{R}\). In particular, for Anosov diffeomorphisms (including transitive ones) the conclusion holds without requiring the non-arithmeticity condition.
\end{teo}

The proof of this result uses some concepts we present below. The first is the notion of asymptotically rational independence for sequences of real numbers, as introduced in \cite{Shaobo} and \cite{Shaobo1}.

\begin{defi}
A sequence $\{a_n\} \subset \mathbb{R}$ is \emph{asymptotically rationally independent} of $b \in \mathbb{R}$ if there exist a positive sequence $\{\epsilon_n\}$ converging to $0$ and integer sequences $\{k_n\}$, $\{l_n\} \subset \mathbb{Z}$ such that
$$0 < |k_n a_n + l_n b| < \epsilon_n.$$
\end{defi}

The next lemma studies the properties of asymptotically rational independence when $\frac{a_n}{b} \in \mathbb{Q}$ and $b \neq 0$. This lemma can be found in \cite[Lemma 2.8]{Shaobo}, but we provide the proof for completeness.

\begin{lem}\label{lem:rational}
Assume that for a sequence $\{a_n\} \subset \mathbb{R}$ and $b \neq 0$, we have $\frac{a_n}{b} \in \mathbb{Q}$ for all $n$. Write $\frac{a_n}{b} = \frac{l_n}{k_n}$ in lowest terms, i.e., with $l_n, k_n \in \mathbb{Z}$, $\gcd(l_n, k_n) = 1$.
\begin{enumerate}
    \item[\emph{(1)}] If $\{a_n\}$ is asymptotically rationally independent of $b$, then $k_n \to \infty$ as $n \to \infty$, and
    $$\inf\,\{ |k a_n + l b| > 0 : k, l \in \mathbb{Z} \} = \Big|\frac{b}{k_n}\Big|.$$
    \item[\emph{(2)}] If $\{a_n\}$ is \emph{not} asymptotically rationally independent of $b$, then there exists a constant $c > 0$ and integers $s_n, t \in \mathbb{Z}$ such that $a_n = c \cdot s_n$ and $b = c \cdot t$ for all $n \in \mathbb{N}$.
\end{enumerate}
\end{lem}

\begin{proof}
(1) By definition, $\{a_n\}$ is asymptotically rationally independent of $b$ if and only if
\begin{equation}\label{Eq-Shaobo 1}
\lim_{n \to \infty} \inf\,\{ |k a_n + l b| > 0 : k, l \in \mathbb{Z} \} = 0.
\end{equation}
Since  $\dfrac{a_n}{b}=\dfrac{l_n}{k_n}$, with  $\text{gcd}(l_n,k_n)=1$, then $\{kl_n+lk_n:k,l\in \mathbb{Z}\}=\mathbb{Z}$, and 
 \begin{equation}\label{Eq-Shaobo 2}
     \{ka_n+lb:k,l\in \mathbb{Z}\}=\Big\{\frac{b}{k_n}(kl_n+lk_n):k,l\in \mathbb{Z}\Big\}=\frac{b}{k_n}\mathbb{Z}.
     \end{equation}

 The result of item (1) it follows directly from (\ref{Eq-Shaobo 1}) and (\ref{Eq-Shaobo 2}).\\
\noindent (2) If $\{a_n\}\subset\re$ is not asymptotically rational independence of $b$, then from (\ref{Eq-Shaobo 1}), then $\displaystyle\inf_{n}\Big|\dfrac{b}{k_n}\Big|>0$, and there is $K\in \mathbb{N}$ such that $|k_n|\leq K$. Thus, $\dfrac{K!}{k_n}\in \mathbb{Z}$ and then 
 $$a_n=\frac{b}{K!}\cdot\frac{K!}{k_n}\cdot l_n.$$
 Define $$s_n = \frac{K!}{k_n} \cdot l_n \cdot \operatorname{sign}(b), \, \, t=K!\cdot\operatorname{sign}(b), \, \, \text{and} \,\, c=\Big|\frac{b}{K!}\Big|,$$ 
 where $\operatorname{sign}(b) = b/|b|$).   \\ A straightforward calculation shows that the sequences and constants defined above satisfy the desired properties. This completes the proof of (\ref{Eq-Shaobo 2}).
\end{proof}
To clarify why the concept of asymptotic rational independence is important for the proof of Theorem~\ref{Shaobo}, we will explain the strategy of the proof (compare with \cite{Shaobo}).

\paragraph{Strategy of the proof of Theorem~\ref{Shaobo}}

\begin{enumerate}
    \item[(1)] Assume that there exist periodic points \( p, q \in \operatorname{Per}(\varphi) \) such that \( S(f, \varphi, p) \) and \( S(f, \varphi, q) \) are rationally independent and satisfy
    \[
    S(f, \varphi, p) < 0 < S(f, \varphi, q).
    \]

    \item[(2)] If condition (1) fails for all pairs of periodic points, we show that there exists a sequence \(\{p_n\}\) converging to \(p\) such that \( S(f, \varphi, p_n) \) is asymptotically rationally independent of \( S(f, \varphi, q) \). If no such sequence exists, then Lemma \ref{lem:rational} implies that \(f\) is arithmetic, yielding a contradiction. 

    \item[(3)] Finally, we show that in both cases, the density property required in Theorem~\ref{Shaobo} holds.
\end{enumerate}

\begin{lem}\emph{\cite[Lemma 3.5]{Shaobo}}
Let $\Lambda$ be a basic set for a diffeomorphism $\varphi$, and let $f\colon \Lambda \to \mathbb{R}$ be a continuous function which is  non-arithmetic. If for every $z \in \textrm{Per}(\varphi)$, the sums $S(f, \varphi, z)$ and $S(f, \varphi, q)$ are rationally dependent, then there exists a sequence $p_n \in \textrm{Per}(\varphi)$ such that $S(f, \varphi, p_n)$ is asymptotically rationally independent of $S(f, \varphi, q)$.
\end{lem}
\begin{proof}
 Note that for a basic set, the set $\textrm{Per}(\varphi)$ is countable. We can therefore enumerate it as $\textrm{Per}(\varphi) := \{p_n\}_{n \geq 1}$. The first statement and the hypotheses imply that for each $n$, there exist integers $-l_n, k_n \in \mathbb{Z} \setminus \{0\}$ with $\gcd(l_n, k_n) = 1$ such that
\[
k_n \cdot S(f, \varphi, p_n) - l_n \cdot S(f, \varphi, q) = 0.
\]
Define $a_n := S(f, \varphi, p_n)$ and $b := S(f, \varphi, q) \neq 0$.

Assume, by contradiction, that such a sequence of periodic points does not exist. Then, by Lemma~\ref{lem:rational}, there exists a constant $c > 0$ such that
\[
\left\{S(f, \varphi, p_n) : n \geq 1 \right\} \subset c \cdot \mathbb{Z}.
\]   
This implies that $f$ is arithmetic, which is a contradiction. 
\end{proof}

In the next we state the proof of Theorem \ref{Shaobo}.
\begin{proof}[\emph{\textbf{Proof of Theorem \ref{Shaobo}}}]\label{Lemma New}
From our hypotheses, since $f$ is dispersed, assume there exist $p, q \in \textrm{Per}(\varphi)$ such that 
$$S(f, \varphi, p)<0<S(f, \varphi, q).$$
Consider now two cases:\\
\indent \textbf{Case 1:} $S(f, \varphi, p)$ and $S(f, \varphi, q)$ are rationally independent. The result then follows directly from \cite[Lemma 3.1 and Remark 3.2]{Shaobo}, as $\Lambda$ is a basic set.\\

\textbf{Case 2:} For every $z \in \textrm{Per}(\varphi)$, the values $S(f, \varphi, z)$ and $S(f, \varphi, q)$ are rationally dependent. In this case, applying Lemma \ref{Lemma New} together with \cite[Lemma 3.3 and Remark 3.4]{Shaobo} completes the proof of the theorem.
Finally, in the case of a transitive Anosov diffeomorphism, if $f$ is arithmetic, then $f$ is cohomologous to a constant and hence concentrated. This yields a contradiction. Therefore, for Anosov systems, the dispersed condition alone suffices to guarantee density (see \cite[Proof of Lemma 3.5]{Shaobo}).
\end{proof}

We immediately obtain consequences for Axiom A diffeomorphisms. 

\begin{coro}\label{Coro1-Shaobo}
Let \(\varphi\) be a \(C^1\) Axiom A diffeomorphism and \(f\) a dispersed and non-arithmetic H\"older continuous function. Then \(\mathcal{B}(f,\varphi)\) is dense in \(\mathbb{R}\). In particular, the result holds for Anosov diffeomorphisms.  In particular, for Anosov diffeomorphisms (including transitive ones) the conclusion holds without requiring the non-arithmeticity condition.
\end{coro}
\begin{proof}
It is just using the Smale's Spectral Decomposition Theorem and Theorem \ref{Shaobo}.
\end{proof}
\subsection{Example of dispersed and arithmetic Lipchitz Function on basic set}\label{example}
The following example demonstrates that the non-arithmeticity condition is necessary for the density conclusion in Theorem \ref{Shaobo}. Thus, the only way to obtain density solely from the dispersed condition is when $\Lambda$ has additional topological properties
Let $\Lambda = \{0,1\}^{\mathbb{Z}}$ with shift $\sigma$, and equip $\Lambda$ with the metric
\[
d(x,y) = \sum_{k \in \mathbb{Z}} 2^{-|k|} |x_k - y_k|.
\]
Consider the function $f\colon \Lambda \to \re$ define by $f(x)=x_0+x_1-1$, where $x=(x_n)_{n\in\mathbb{Z}}$.\\
The function $f$ has the following properties 
\begin{enumerate}
    \item \textbf{$f$ is Lipschitz continuity (in particular H\"older).}\\
    For any $x, y \in \Lambda$:
\begin{eqnarray*}
    |\varphi(x) - \varphi(y)|& = &|(x_0 + x_1 - 1) - (y_0 + y_1 - 1)| = |(x_0 - y_0) + (x_1 - y_1)|\\
    &\leq & d(x,y)+2d(x,y)=3d(x,y).
\end{eqnarray*}

\item \textbf{Not cohomologous to any constant and dispersed.}\\
Suppose, for contradiction, that $\varphi$ is cohomologous to a constant $c \in \mathbb{R}$.
That is, there exists a continuous function $\eta : \Lambda \to \mathbb{R}$ such that
\[
f(x) = c + \eta(\sigma( x)) - \eta(x) \quad \text{for all } x \in \Lambda.
\]

Consider the two fixed points:
\[
e_0 = (\dots, 0, 0, 0, \dots), \quad e_1 = (\dots, 1, 1, 1, \dots).
\]

Observe that $f(e_0) = -1$, and $f(e_1)=1$. Therefore, $f$ is dispersed.\\
Since $\sigma(e_0)=e_0$ and $\sigma(e_1)=e_1$, then the coboundary condition gives:
\[
-1 = c + \eta(e_0) - \eta(e_0) = c \,\,\, \text{and} \,\,\, 1=c + \eta(e_1) - \eta(e_1) = c
\]
Thus, we obtain $c = -1 = 1$, a contradiction. Therefore $\varphi$ cannot be cohomologous to any constant $c \in \mathbb{R}$.
\item \textbf{$f$ is arithmetic.}\\
Let $x$ be a periodic point with minimal period $p \ge 1$ and repeating block
\[
(a_0, a_1, \dots, a_{p-1}), \quad a_k \in \{0,1\}.
\]
That is, $x_k = a_{k \bmod p}$ for all $k \in \mathbb{Z}$.

The Birkhoff sum over one period is:
\begin{eqnarray*}
    S(f,\sigma, x) &=& \sum_{k=0}^{p-1} f(\sigma^k x)=\sum_{k=0}^{p-1} \big( a_k + a_{k+1 \bmod p} - 1 \big)\\
    &=& 2\sum_{k=0}^{p-1} a_k - p\in \mathbb{Z}, 
\end{eqnarray*}
since $\displaystyle\sum_{k=0}^{p-1} a_{k+1 \bmod p} = \sum_{k=0}^{p-1} a_k.$
\end{enumerate}


\subsection{Some Consequences of Theorem \ref{Shaobo}}\label{measure spectra}
\noindent The next result describes the  \emph{average Birkhoff spectrum},  given by
\[
\mathcal{AB}(f,\varphi)=\Big\{\frac{S(f, \varphi, p)}{\tau_p}: p\in \operatorname{Per}(\varphi)\Big\}.
.\] When the dynamics are restricted to an invariant subset \(\Lambda\), the corresponding set is denoted by \(\mathcal{AB}(f,\varphi, \Lambda)\).

Let
\[
m(f):=\inf_{\mu\in \mathcal{M}_1(\Lambda, \varphi)}\int f\,d\mu, \qquad
M(f):=\sup_{\mu\in \mathcal{M}_1(\Lambda, \varphi)}\int f\,d\mu,
\]
where \(\mathcal{M}_1(\Lambda, \varphi)\) denotes the set of \(\varphi\)-invariant probability measures supported on \(\Lambda\).
\begin{teo}\label{main2}  
Let \(\Lambda\subset U\) be a basic set for \(\varphi \colon U \to M\) and let \(f \colon U \to \mathbb{R}\) be H\"older continuous. Then
\[
\Big\{\int f\,d\mu : \mu\in \mathcal{M}_{1}(\Lambda, \varphi)\Big\}=[m(f), M(f)].
\]
\end{teo}
Although this statement can be derived from the connectedness of the space of ergodic measures \cite{Sigmund2}, we give a proof based on techniques from Theorem \ref{Shaobo}.

The proof of Theorem \ref{main2} will be a consequence of following two Lemmas.

\begin{lem}\label{Lmain2*}
Let $\Lambda$ be a basic set for a diffeomorphism $\varphi$, and let \(f \colon \Lambda \to \mathbb{R}\) be continuous and not cohomologous to constant. For any interval $I\subset \re$, there exists \(\beta \in I\) such that $f-\beta$ is non-arithmetic.
\end{lem}

\begin{proof}
If $p\in \operatorname{Per}(\varphi)$, set \(A_p:=\dfrac{S(f,\varphi,p)}{\tau_p}\). Since $f$ is not cohomologous to a constant there are \(p_1\neq p_2\) with \(A_{p_1}\neq A_{p_2}\).  
For \(\beta\neq A_{p_2}\) define
\[
R(\beta):=\frac{A_{p_1}-\beta}{A_{p_2}-\beta}.
\]
For each rational \(\frac{m}{n}\in\mathbb{Q}\setminus \{0\}\), the equation \(R(\beta)=\frac{m}{n}\) has at most one solution \(\beta\). Hence
\[
B:=\{\,\beta\in I\mid R(\beta)\in\mathbb{Q}\,\}
\]
is countable. Pick \(\beta\notin B\cup\{A_{p_2}\}\). Then \(R(\beta)\notin\mathbb{Q}\).

\noindent Now, put
\[
u:=\tau_{p_1}(A_{p_1}-\beta)=S(f-\beta,\varphi,p_1),\qquad
v:=\tau_{p_2}(A_{p_2}-\beta)=S(f-\beta,\varphi,p_2).
\]
Their ratio 
\(\ds
\frac{u}{v}= \frac{\tau_{p_1}}{\tau_{p_2}}\,R(\beta)\), is irrational.

Assume, for a contradiction, that $f-\beta$ is arithmetic, then for some \(c\neq0\) the set of all such Birkhoff sums is contained in \(c\mathbb{Z}\). Then \(u=c k_1\) and \(v=c k_2\) for integers \(k_1,k_2\), so \(u/v=k_1/k_2\in\mathbb{Q}\), contradicting the irrationality of \(u/v\).\\
\noindent Thus, no such \(c\) exists, which proves the lemma.
\end{proof}

For the next lemma, we use the following Sigmund's Theorem related to the density of periodic measures.
\begin{teo}[\cite{Sigmund1}]\label{Sigmund}
    If $\Lambda$ is a basic set for a diffeomorphism $\varphi$, then the set \\ $\displaystyle \Big\{\frac{1}{\tau_p}\sum_{0}^{\tau_p-1}\delta_{\varphi^{i}(p)}: z\in \emph{Per}(\varphi)\Big\}$ is dense in $\mathcal{M}(\Lambda, \varphi)$ with the topology $*$-weak.
\end{teo}
Using Theorem \ref{Sigmund} and Lemma \ref{Lmain2*}, we have

\begin{lem}\label{Lmain2}
    Let $\Lambda \subset U$ be a basic set for a flow $\varphi\colon U \to M$, and let $f\colon U \to \mathbb{R}$ be a H\"older continuous function. Then, the set $\mathcal{AB}(f, \varphi, \Lambda)$ is dense in the interval $[m(f), M(f)]$.
\end{lem}

\begin{proof}

We begin by stating the following identities:
\[
m(f) = \inf_{p \in \operatorname{Per}(\varphi)} \frac{S(f, \varphi, p)}{\tau_p} \quad \text{and} \quad M(f) = \sup_{p \in \operatorname{Per}(\varphi)} \frac{S(f, \varphi, p)}{\tau_p},
\]
where \( m(f) = \inf_{\mu \in \mathcal{M}_1(\Lambda, \varphi)} \int f\,d\mu \) and \( M(f) = \sup_{\mu \in \mathcal{M}_1(\Lambda, \varphi)} \int f\,d\mu \).

It is immediate from the definitions that
\[
m(f) \leq \inf_{p \in \operatorname{Per}(\varphi)} \int f\,d\mu_p \quad \text{and} \quad M(f) \geq \sup_{p \in \operatorname{Per}(\varphi)} \int f\,d\mu_p,
\]
where \( \mu_p = \frac{1}{\tau_p} \sum_{i=0}^{\tau_p-1} \delta_{\varphi^i(p)} \).

To prove the reverse inequality for \( m(f) \), let \( \{\mu_i\} \) be a sequence of invariant measures such that \( \int f \,d\mu_i \to m(f) \). By compactness, we may assume that \( \mu_i \to \mu \) weakly for some invariant measure \( \mu \), and it follows that \( \int f \,d\mu = m(f) \).

By Theorem \ref{Sigmund}, for each \( \mu_i \) there exists a sequence of periodic points \( p_n^i \in \operatorname{Per}(\varphi) \) such that \( \mu_{p_n^i} \to \mu_i \) as \( n \to \infty \). A standard diagonal argument allows us to find a subsequence \( p_{n_i}^i \) such that \( \mu_{p_{n_i}^i} \to \mu \). In particular,
\[
\int f \,d\mu_{p_{n_i}^i} \to m(f).
\]
Consequently,
\[
\inf_{p \in \operatorname{Per}(\varphi)} \frac{S(f, \varphi, p)}{\tau_p} \leq \lim_{i \to \infty} \int f \,d\mu_{p_{n_i}^i} = m(f),
\]
which establishes the desired identity for \( m(f) \). The proof for \( M(f) \) is analogous.

Now let us prove the density:
    If the interval $[m(f), M(f)]$ consists of a single point, i.e., $m(f) = M(f)$, the proof is trivial. We may therefore assume $m(f) < M(f)$ (in particular, $f$ is not cohomnologous to constant).

    Let $a, b$ be such that $m(f) < a < b < M(f)$. We will prove that there exists a periodic point $p \in \text{Per}(\varphi)$ for which the average $\frac{S(f, \varphi, p)}{\tau_p}$ lies in $I:=(a, b)$.

    Suppose, by contradiction, that no such periodic point exists. We then define the following two subsets of $\text{Per}(\varphi)$:
    \[
    \Gamma_a = \left\{ p \in \text{Per}(\varphi) : \frac{S(f, \varphi, p)}{\tau_p} < a \right\}
    \quad \text{and} \quad
    \Pi_b = \left\{ p \in \text{Per}(\varphi) : \frac{S(f, \varphi, p)}{\tau_p} > b \right\}.
    \]
    By our assumption, $\text{Per}(\varphi) = \Gamma_a \cup \Pi_b$.

    For every $\beta \in I$ the function $f-\beta$ is dispersed. Moroever, from Lemma \ref{Lmain2*} there is $\beta\in I$ such that $f-\beta$ is non-arithmetic, then by Theorem~\ref{Shaobo}, the set $\mathcal{B}(f - \beta, \varphi, \Lambda)$ is dense in $\mathbb{R}$. Consider the set
    \[
    \Delta_\beta = \left\{ p \in \text{Per}(\varphi) : S(f, \varphi, p) - \beta \tau_p \in (a, b) \right\}.
    \]
    By density, the set $\Delta_\beta$ is infinite due to the density property.

    We now examine the intersection of $\Delta_\beta$ with the sets $\Gamma_a$ and $\Gamma_b$.

    \begin{itemize}
        \item Let $p \in \Gamma_a \cap \Delta_\beta$. Then:
        \[
        a < S(f, \varphi, p) - \beta \tau_p < a \tau_p - \beta \tau_p = (a - \beta)\tau_p.
        \]
        If $a \geq 0$, this is an immediate contradiction because $a - \beta < 0$. If $a < 0$, it follows that $\tau_p < \frac{a}{a - \beta}$, implying the period $\tau_p$ is bounded above. Consequently, the set $\Gamma_a \cap \Delta_\beta$ is finite.

        \item Now, let $p \in \Pi_b \cap \Delta_\beta$. Then:
        \[
        b \tau_p - \beta \tau_p < S(f, \varphi, p) - \beta \tau_p < b.
        \]
        Since $b-\beta>0$, this implies $\tau_p < \frac{b}{b - \beta}$, so again we have a contradiction (if $b<0$)  or the period $\tau_p$ is bounded (if $b>0$). Hence, the set $\Pi_b \cap \Delta_\beta$ is also finite.
    \end{itemize}

    However, we have $\Delta_\beta = (\Gamma_a \cap \Delta_\beta) \cup (\Pi_b \cap \Delta_\beta)$. Since the union on the right-hand side is finite, this contradicts the earlier conclusion that $\Delta_\beta$ is infinite. This contradiction completes the proof.
\end{proof}

 \begin{proof}[\emph{\textbf{Proof of Theorem \ref{main2}}}]
Assume, without loss of generality, that $m(f) < M(f)$, and choose $a,\tilde{a}, b,\tilde{b}$ such that $m(f) < a < \tilde{a}<\tilde{b}<b < M(f)$. By Lemma \ref{Lmain2}, there exists $p_k \in \text{Per}(\varphi)$ such that  
\[
\int f \, d\mu_k = \frac{S(f, \varphi, p_k)}{\tau_{p_k}} \in (\tilde{a}, \tilde{b}),
\]  
where  
$
\displaystyle\mu_k = \frac{1}{\tau_{p_k}} \sum_{i=0}^{\tau_{p_k}-1} \delta_{\varphi^{i}(p_k)}
$.
If necessary, pass to a subsequence so that $\mu_k \to \mu \in \mathcal{M}(\Lambda, \varphi)$. Then  
\[
\int f \, d\mu_k \to \int f \, d\mu \in [\tilde{a}, \tilde{b}]\subset (a,b),
\]  
which completes the proof.
 \end{proof}
\subsection{Proof of Theorem \ref{mainT}}\label{Proof mainT}
Theorem~\ref{Shaobo} and Liv\v{s}ic's Theorem will be the main tools in the proof of Theorem~\ref{mainT}. Furthermore, we will use Theorem~\ref{Shaobo} to demonstrate that the solutions obtained in \cite{ZouWei} for the finite approximation of Liv\v{s}ic's Theorem are trivial.
 \begin{proof}[\emph{\textbf{Proof of Theorem \ref{mainT}}}]
From Liv\v{s}ic's Theorem, it is clear that (ii) is equivalent to (iii). It is also evident that either (ii) or (iii) implies (i). Therefore, our task reduces to proving that (i) implies (ii). To prove this, we will demonstrate the contrapositive: we assume that $f$ is not cohomologous to $0$, and prove that the set $\mathcal{B}(f,\varphi, \Lambda)$ is unbounded.


\noindent \textbf{Case 1:} For all $p,q\in\text{Per}(\varphi)$,
$$\frac{S(f, \varphi, p)}{\tau_p}=\frac{S(f, \varphi, q)}{\tau_q}=\alpha.$$
Since $f$ is not cohomologous to $0$, it follows that $\alpha\neq 0$. Consequently,
$$\mathcal{B}(f,\varphi, \Lambda)=\{\alpha\tau_p: p\in \text{Per}(\varphi)\}.$$
This set is unbounded because the set of periods $\{\tau_p\}$ for periodic orbits in a basic set is itself unbounded.

\noindent \textbf{Case 2:} There exist $p,q\in\text{Per}(\varphi)$ such that
$$\frac{S(f, \varphi, p)}{\tau_p}<\frac{S(f, \varphi, q)}{\tau_q}.$$
Consider the interval $I:=\Big(\frac{S(f, \varphi, p)}{\tau_p}, \frac{S(f, \varphi, q)}{\tau_q} \Big)$. From Lemma \ref{Lmain2*}, we can choose a constant $\beta\in I$  such that the H\"older continuous function $g(x)=f(x)-\beta$ is dispersed and non-arithmetic. Therefore,  
by Theorem \ref{Shaobo}, the set $\mathcal{B}(g,\varphi, \Lambda)$ is dense in $\mathbb{R}$. Therefore, for any $M>0$, there exists a periodic point $p$ such that
$$M < S(g, \varphi, p) = S(f, \varphi, p) - \beta\tau_p \leq M+1.$$
This inequality implies that $S(f, \varphi, p) > M + \beta\tau_p > M$. Hence, $\mathcal{B}(f,\varphi, \Lambda)$ is unbounded, which completes the proof.
\end{proof}
In \cite{ZouWei} the author studied finite approximation of Liv\v sic Theorem, more specifically  
\begin{teo}\emph{\cite[Theorem 1.1]{ZouWei}}\label{thm:ZouWei}
Let \(\varphi\) be a transitive Anosov diffeomorphism of a compact manifold \(M\) and let \(0 < \alpha \leq 1\). There exist constants \(0 < \beta \leq \alpha\), \(C > 0\), and \(\tau > 0\) such that for any \(\varepsilon > 0\) and any \(\ f \in C^{\alpha}(M)\) with \(\|f\|_{C^{\alpha}} \leq 1\) satisfying
\[
\Big|\sum_{i=0}^{n-1} f(\varphi^i(p))\Big| \leq \varepsilon, \quad \forall p = \varphi^n(p) \text{ with } n \leq \varepsilon^{-\frac{1}{2}},
\]
there exist functions \(u \in C^{\beta}(M)\) and \(h \in C^{\beta}(M)\) such that
\[
f(x) = u(\varphi(x)) - u(x) + h(x), \quad \forall x \in M.
\]
Moreover,
\[
\|u\|_{C^{\beta}} \leq C \quad \text{and} \quad \|h\|_{C^{\beta}} \leq C\varepsilon^{\tau}.
\]
\end{teo}
As a direct application of Theorem~\ref{mainT}, we show that all solutions $h(x)$ are trivial, \emph{i.e.}, $h(x)=0$ for all $x\in M$. 

\noindent Indeed, fix $0<\epsilon\leq 1$ and consider the subset of periodic points
\[
\Gamma_{\epsilon}(\varphi)=\{p\in \operatorname{Per}(\varphi): \tau_{p}\leq \epsilon^{-\frac{1}{2}}\}.
\]
Since every periodic point $p\in\operatorname{Per}(\varphi)$ has period $\tau_p\geq 1$, then there is $0<\epsilon\leq 1$ such that $\tau_p\leq \epsilon^{-\frac{1}{2}}$. Therefore
\[
\operatorname{Per}(\varphi)=\bigcup_{0<\epsilon\leq 1}\Gamma_{\epsilon}(\varphi).
\]
Now, the hypothesis of Theorem~\ref{thm:ZouWei} implies that for each $\epsilon>0$ and all $p\in\Gamma_{\epsilon}(\varphi)$,
\[
\Big|\sum_{i=0}^{\tau_p-1} f(\varphi^i(p))\Big| \leq \epsilon.
\]
We conclude that
\[
\Big|\sum_{i=0}^{\tau_p-1} f(\varphi^i(p))\Big| \leq 1 \quad \text{for all } p \in \operatorname{Per}(\varphi),
\]
which is equivalent to $\mathcal{B}(f,\varphi)\subset [-1,1]$.

If $\mathcal{B}(f,\varphi)$ contained both positive and negative values---that is, if 
$\mathcal{B}(f,\varphi)\cap [-1,0) \neq \emptyset$ and $\mathcal{B}(f,\varphi)\cap (0,1] \neq \emptyset$---then Corollary~\ref{Coro1-Shaobo} would imply that $\mathcal{B}(f,\varphi)$ is dense in $\mathbb{R}$, contradicting its boundedness. 

Therefore, $\mathcal{B}(f,\varphi)$ must lie entirely in $[0,1]$ or in $[-1,0]$. In either case, Theorem~\ref{mainT} applies and shows that $f$ is cohomologous to zero, and hence $h(x)=0$ for all $x\in M$.

\subsection{Proof of Theorem \ref{Rigidity}}
The proof of Theorem~\ref{Rigidity} is rather extensive. To complete it, we need several lemmas from number theory concerning upper density and rational independence.

\subsubsection{Number-theoretic lemmas}

\begin{lem}\label{Lemma4-Rigidity}
Let \(a, b > 0\) with \(a/b \notin \mathbb{Q}\), and let \(c_1, \dots, c_n > 0\). Then there exists \(\beta \in (0,1) \cap \mathbb{Q}\) such that for all \(k = 1, \dots, n\),  
\[
\frac{a\beta + b}{c_k} \notin \mathbb{Q}.
\]
\end{lem}

\begin{proof}
Define for each \(k = 1, \dots, n\):
\[
B_k = \left\{ \beta \in (0,1) \cap \mathbb{Q} : \frac{a\beta + b}{c_k} \in \mathbb{Q} \right\}.
\]
Equivalently, \(B_k = \left\{ \beta \in (0,1) \cap \mathbb{Q} : a\beta + b \in \mathbb{Q} \cdot c_k \right\}\).

We will show that \(\bigcup_{k=1}^n B_k \subsetneq (0,1) \cap \mathbb{Q}\).

\medskip

\noindent \textbf{Step 1:} \textit{Each \(B_k\) is a proper subset of \((0,1) \cap \mathbb{Q}\).}

Suppose for contradiction that \(B_k = (0,1) \cap \mathbb{Q}\) for some \(k\).  
Pick \(\beta_1, \beta_2 \in (0,1) \cap \mathbb{Q}\), \(\beta_1 \ne \beta_2\). Then  
\[
a\beta_1 + b = q_1 c_k, \quad a\beta_2 + b = q_2 c_k, \quad q_1, q_2 \in \mathbb{Q}.
\]
Subtracting: \(a(\beta_1 - \beta_2) = (q_1 - q_2)c_k\), so  
\[
c_k = \frac{a(\beta_1 - \beta_2)}{q_1 - q_2} \in \mathbb{Q} \cdot a.
\]
Let \(c_k = t a\) with \(t \in \mathbb{Q}\). Then from \(a\beta_1 + b = q_1 t a\), we get  
\[
b = a(q_1 t - \beta_1) \in \mathbb{Q} \cdot a.
\]
So \(a, b \in \mathbb{Q} \cdot a\) trivially, and \(b = r a\) with \(r \in \mathbb{Q}\), hence \(a/b = 1/r \in \mathbb{Q}\), contradiction.  
Thus \(B_k \subsetneq (0,1) \cap \mathbb{Q}\).

\medskip

\noindent \textbf{Step 2:} \textit{No finite union \(B_1 \cup \dots \cup B_n\) equals \((0,1) \cap \mathbb{Q}\).}

Suppose for contradiction that  
\[
B_1 \cup \dots \cup B_n = (0,1) \cap \mathbb{Q}.
\]
Since \((0,1) \cap \mathbb{Q}\) is infinite, by the pigeonhole principle, some \(B_{k_0}\) contains infinitely many distinct rationals in \((0,1)\).

Pick \(\beta_1, \beta_2 \in B_{k_0}\), \(\beta_1 \ne \beta_2\). Then  
\[
a\beta_1 + b = q_1 c_{k_0}, \quad a\beta_2 + b = q_2 c_{k_0}, \quad q_1, q_2 \in \mathbb{Q}.
\]
Subtracting: \(a(\beta_1 - \beta_2) = (q_1 - q_2)c_{k_0}\), so  
\[
c_{k_0} = \frac{a(\beta_1 - \beta_2)}{q_1 - q_2} \in \mathbb{Q} \cdot a.
\]
Let \(c_{k_0} = t a\), \(t \in \mathbb{Q}\). Then from \(a\beta_1 + b = q_1 t a\),  
\[
b = a(q_1 t - \beta_1).
\]
Thus \(b \in \mathbb{Q} \cdot a\), so \(a/b \in \mathbb{Q}\), contradiction.

\noindent Hence the union \(\bigcup_{k=1}^n B_k\) is a proper subset of \((0,1) \cap \mathbb{Q}\).

\medskip

Choose \(\beta \in (0,1) \cap \mathbb{Q} \setminus \bigcup_{k=1}^n B_k\).  
Then for all \(k\), \(\beta \notin B_k\), so \(a\beta + b \notin \mathbb{Q} \cdot c_k\), hence \(\frac{a\beta + b}{c_k} \notin \mathbb{Q}\).
\end{proof}

\begin{lem}\label{LEMMA3NEW-RIG}
Let $A_1, A_2, \dots, A_m \subseteq \mathbb{N}$ and suppose $\ds \bigcup_{i=1}^m A_i \supseteq \{n \in \mathbb{N} : n \geq n_0\}$ for some $n_0 \in \mathbb{N}$. Then there exists $i_0$ such that
$${d}^*(A_{i_0}) \geq \frac{1}{m},$$
where ${d}^{*}(A) = \displaystyle\limsup\limits_{N \to \infty} \frac{|A \cap \{1,\dots,N\}|}{N}$ is the upper density.
\end{lem}
\begin{proof}   
Let $B = \bigcup_{i=1}^m A_i$. For $N > n_0$,
\[
|B \cap [1,N]| \geq N - n_0 + 1.
\]
so ${d}^*(B) = 1$. 

\noindent Now for each $N$, $\sum_{i=1}^m |A_i \cap [1,N]| \geq |B \cap [1,N]|$, consequently 
\[
\sum_{i=1}^m {d}^{*}(A_i) \geq {d}^*(B) = 1.
\]
Therefore, there must exist $i_0$ such that 
$
{d}^{*}(A_{i_0}) \geq \frac{1}{m}.
$
\end{proof}

\begin{lem}\label{LEMMA2NEW-RIG}
Let \(\alpha \in \mathbb{R}\setminus\mathbb{Q}\) be irrational, $\theta\in \re$  and let \(A \subset \mathbb{N}\) with $d^*(A)=\gamma >0$.
Let \(S_A:=\{\{\alpha n+\theta \}: n \in A\}\subset[0,1]\), where \(\{x\}=x-\lfloor x \rfloor\) denotes the fractional part of $x$.
Then there exists \(x\in S_A\) such that
\[
\frac{\gamma}{4}\le x\le 1-\frac{\gamma}{4}.
\]
\end{lem}

\begin{proof}
Suppose, for contradiction, that every \(x\in S_A\) satisfies \(x<\gamma/4\) or \(x>1-\gamma/4\). 
Then for all \(n\in A\),
\[
\{\alpha n+\theta \}\in[0,\gamma/4)\cup(1-\gamma/4,1).
\]
\noindent Since \(\alpha\) is irrational, the sequence \((\{\alpha n+\theta\})_{n=1}^\infty\) is uniformly distributed mod~1. Hence
\[
\lim_{N\to\infty}\frac{|\{n\le N:\{\alpha n+\theta\}\in[0,\gamma/4)\cup(1-\gamma/4,1)\}|}{N}=2\gamma/4=\gamma/2.
\]
For each \(N\) we have
\[
A\cap\{1,\dots,N\}\subset\{n\le N:\{\alpha n+\theta\}\in[0,\gamma/4)\cup(1-\gamma/4,1)\},
\]
so
\[
\frac{|A\cap\{1,\dots,N\}|}{N}\le\frac{|\{n\le N:\{\alpha n+\theta\}\in[0,\varepsilon)\cup(1-\varepsilon,1)\}|}{N}.
\]
Therefore,  \(d^*(A)\le 2\gamma/4=\gamma/2< \gamma = d^*(A)\) a contradiction.\\
Thus there exists \(n_0\in A\) with \(\gamma/4 \le x:=\{\alpha n_0+\theta\}\le 1-\gamma/4\) as required.

\end{proof}
We now show that when \(T\subset \mathbb{N}\) has positive upper density, the expression 
\(c G(n,m) - (a m + b n)\) with \(m = \lfloor \beta n \rfloor\) for general natural function $G\colon \mathbb{N}^2\to \mathbb{N}$ cannot be made to cluster arbitrarily closely around any single real number$-$there is a  minimum spread that depends on the density of \(T\) and the coefficient \(c\). This dispersion property, captured in the following lemma, will later prevent certain undesirable configurations in the proof of Theorem \ref{Rigidity}.

\begin{lem}\label{Lemma3-Rigidity}
Let \( a, b, c > 0 \), \(a/b \notin \mathbb{Q}\), and  \( \beta \in (0,1) \cap \mathbb{Q}\) as in \emph{Lemma \ref{Lemma4-Rigidity} \text{(}for $a,b$, $c$\text{)}} and \( G: \mathbb{N}^2 \to \mathbb{N} \). Let $T\subset \mathbb{N}$ with $d^{*}(T)=\gamma >0$. Define  
\[
S = \{ c G(n,m) - (a m + b n) : m = \lfloor \beta n \rfloor,\ n \in T \}.
\]  
Then for any \( A \in \mathbb{R} \) and any $\delta$ with $0<\delta<\delta_0:=\frac{\gamma c}{4}$
\[
S \not\subseteq [A-\delta, A+\delta].
\]
\end{lem}

\begin{proof}
Let \( m = \lfloor \beta n \rfloor \) and \( A > 0 \). Define:
\[
B_n = c G(n, \lfloor \beta n \rfloor) - (a \lfloor \beta n \rfloor + b n) - A.
\]
Let \( k_n = G(n, \lfloor \beta n \rfloor) \in \mathbb{N} \). Then:
\[
B_n = c k_n - a \lfloor \beta n \rfloor - b n - A.
\]
Writing \( \lfloor \beta n \rfloor = \beta n + \varepsilon_n \), where \( \varepsilon_n = \lfloor \beta n \rfloor - \beta n \in (-1, 0] \), we obtain:
\[
B_n = c k_n - a(\beta n + \varepsilon_n) - b n - A = c k_n - (a \beta + b) n - a \varepsilon_n - A.
\]
Let \( \theta= -(a \beta + b) \). Then:
\[
B_n = c k_n + \theta n - a \varepsilon_n - A.
\]

\noindent Since \( k_n \in \mathbb{Z} \), we have modulo 1:
\[
\frac{B_n}{c} \equiv \frac{\theta}{c} n - \frac{a}{c} \varepsilon_n - \frac{A}{c} \pmod{1}.
\]
Define
\[
X_n = \frac{\theta}{c} n - \frac{a}{c} \varepsilon_n - \frac{A}{c}.
\]
The next step is to analyze \( X_n = \frac{\theta}{c} n - \frac{a}{c} \varepsilon_n - \frac{A}{c} \pmod{1} \).

\noindent Since \( \beta \) is rational, the sequence \( \{\beta n\} = \beta n - \lfloor \beta n \rfloor \) takes on only finitely many values in \( [0,1) \). Consequently, the sequence \( -\frac{a}{c}\varepsilon_n - \frac{A}{c} = \frac{a}{c}\{\beta n\} - \frac{A}{c} \) also takes on finitely many values, contained in the interval \( [-\frac{A}{c}, a - \frac{A}{c}) \).

\noindent Given that \( \frac{\theta}{c} = -\frac{a\beta + b}{c} \) is irrational, Lemma \ref{LEMMA2NEW-RIG} implies that there is $n_0\in T$ such that 
\[
\frac{\gamma}{4}\le X_{n_0} \, \, (\text{mod}\,\,  1) \le 1-\frac{\gamma}{4}.
\]
In particular, $\mathrm{dist}\Big(\frac{B_{n_0}}{c}, \mathbb{Z}\Big)=\text{dist}\Big ( X_{n_0}, \mathbb{Z}\Big)\geq \dfrac{\gamma}{4}$.\\
Now, for the sake of contradiction, suppose \( S \subseteq [A - \delta, A + \delta] \), for some $\delta<\dfrac{\gamma c}{4}$. This would imply \( \mathrm{dist}(\frac{B_n}{c}, \mathbb{Z}) < \dfrac{\delta}{c} \) for all \( n \). Which implies that $\dfrac{\gamma}{4}<\dfrac{\delta}{c} $. This implies a contradiction for the choice of \( \delta_0 = \dfrac{\gamma c}{4} \). It follows that for any \( \delta < \delta_0 \), the set \( S \) cannot be a subset of \( [A - \delta, A + \delta] \), which is the desired result.
\end{proof}



To conclude this subsection, we prove Theorem~\ref{Rigidity}. The proof leverages the hyperbolic structure of the basic set---using tools such as the shadowing lemma$---$to establish the appropriate setting for applying the preceding lemmas. It is worth noting that the proof follows a similar overall approach to \cite{Shaobo}, but we provide full details here because we need to produce certain specific combinations of Birkhoff sums for our argument. Moreover, some of these combinatorial structures will be needed again in the proof of Theorem~\ref{mainT2}. 

It is important to note, the proof also take advance in similar approch like \cite{Shabao}, but write all the details since we need some specfic combination/phenomenon. Moreover, some of this combinations will be necessary in the   proof of Theorem \ref{mainT2}

\begin{proof}[\emph{\textbf{Proof of Theorem \ref{Rigidity}}}]
First, if $f$ is dispersed, then by Theorem~\ref{Shaobo} and the fact that by hypotheses $f$ is non-arithmetic, the set $\mathcal{B}(f,\varphi,\Lambda)$ is dense in $\mathbb{R}$.  
But $\mathcal{B}(f,\varphi,\Lambda) \subset \bigcup_i a_i \mathbb{Z}$ is a countable union of discrete additive subgroups, hence nowhere dense in $\mathbb{R}$; this contradicts density in $\mathbb{R}$.  
Therefore, $f$ cannot be dispersed. Consequently, we may assume $f$ is not dispersed, i.e.
\[
\mathcal{B}(f,\varphi,\Lambda) \subset \bigcup_{i=1}^{l} a_i\mathbb{N}.
\]
\noindent The proof of this theorem relies on the fact that for some $i \neq j$ the ratio $a_i/a_j$ is irrational.\\
\ \\
\noindent \textbf{Claim:} There are $i\neq j$ such that  $\dfrac{a_i}{a_j}\notin Q$. 
\vspace{-0.3cm}
\begin{proof}[\emph{\textbf{Proof of Claim}}]
Assume, for the sake of contradiction, that all ratios $\dfrac{a_i}{a_j}$ are rational. 
Write each ratio $\dfrac{a_i}{a_{i+1}}$ in lowest terms as $\dfrac{p_i}{q_i}$, for $i=1,\dots l-1$, then

$$\mathcal{B}(f, \varphi, \Lambda)\subset \bigcup_{i=1}^{l}a_i\mathbb{N}\subset \frac{a_l}{\prod_{i=1}^{l-1}q_i}\mathbb{N}$$

Thus, we have a contradiction since $f$ is non-arithmetic. Hence the assumption that all ratios 
$\dfrac{a_i}{a_j}$   are rational is false, which completes the proof of the claim.
\end{proof}

\noindent From the Claim and the assumption that $f$ is non-arithmetic,  we can assume that there are two different indices $i_1, i_2$ such that $\dfrac{a_{i_1}}{a_{i_2}} \notin \mathbb{Q}$.

Choose periodic points $p,q$ with $2S(f,\varphi,p) = a_{i_1}n(p) := a$ and $2S(f,\varphi,q) = a_{i_2}n(q) = b$, where $n(p),n(q) \in \mathbb{N}$. Then $a/b \notin \mathbb{Q}$.

Since $\Lambda$ is a basic set, $p$ and $q$ are homoclinically related (see Remark \ref{Homoclinic}). Thus there exist points $x,y\in \Lambda$ such that 
\begin{equation}\label{Eq1:Rig}
    x \in W^s(p) \pitchfork W^u(q) \quad \text{and} \quad y \in W^s(q) \pitchfork W^u(p).
\end{equation}

\begin{center}
\usetikzlibrary{decorations.markings,arrows.meta}
\begin{tikzpicture}[
    arrow in middle/.style={
        decoration={
            markings,
            mark=at position 0.5 with {\arrow[line width=0.8pt, scale=1.5]{Latex[width=4pt, length=3pt]}}
        },
        postaction={decorate}
    }
]

\coordinate (A) at (0,0);
\coordinate (B) at (6,0);
\coordinate (C) at (6,4);
\coordinate (D) at (0,4);
\coordinate (A1) at (1,-1);
\coordinate (B1) at (5,-1);
\coordinate (C1) at (5,5);
\coordinate (D1) at (1,5);

\draw[green!30!black, line width=1.5pt, arrow in middle] (A) .. controls (3,0.3) .. (B) 
    node[midway, below] {$W^u(p)$};

\draw[green!30!black, line width=1.5pt, arrow in middle] (C) .. controls (3,4.3) .. (D) 
    node[midway, above] {$W^u(q)$};

\draw[blue!30!black, line width=1.5pt, arrow in middle] (B1) .. controls (4.6,2) .. (C1) 
    node[midway, right] {$W^s(q)$};

\draw[blue!30!black, line width=1.5pt, arrow in middle] (D1) .. controls (0.6,2) .. (A1) 
    node[midway, left] {$W^s(p)$};

\fill[red!60!black] (0.88,0.1) circle (2pt);
\node[below left] at (0.88,0.1) {$p$};

\fill[red!60!black] (0.78,0.8) circle (2pt);
\node[left] at (0.78,0.8) {\tiny{$\varphi^{L_4-1}(x)$}};

\fill[red!60!black] (0.78,3.4) circle (2pt);
\node[left] at (0.79,3.4) {\tiny{$\varphi(x)$}};
\fill[red!60!black] (4.79,0.8) circle (2pt);
\node[right ] at (4.8,0.8) {\tiny{$\varphi(y)$}};

\fill[red!60!black] (4.8,3.4) circle (2pt);
\node[right] at (4.8,3.4) {\tiny{$\varphi^{L_2-1}(y)$}};
\fill[red!60!black] (1.6,4.15) circle (2pt);
\node[above] at (1.6,4.2) {\tiny{$\varphi^{-L_3}(x)$}};

\fill[red!60!black] (4.15,4.15) circle (2pt);
\node[above] at (4.1,4.2) {\tiny{$\varphi^{-1}(x)$}};

\fill[red!60!black] (1.6,0.15) circle (2pt);
\node[below] at (1.6,0.15) {\tiny{$\varphi^{-L_1}(y)$}};

\fill[red!60!black] (4.15,0.15) circle (2pt);
\node[below] at (4.1,0.15) {\tiny{$\varphi^{-1}(y)$}};


\fill[red!60!black] (4.88,4.1) circle (2pt);
\node[above right] at (4.88,4.1) {$q$};

\fill[red!60!black] (4.88,0.1) circle (2pt);
\node[below right] at (4.88,0.1) {$y$};

\fill[red!60!black] (0.88,4.1) circle (2pt);
\node[above left] at (0.88,4.1) {$x$};

\node[below] at (current bounding box.south) {\textbf{Fig. 1:} Homoclinic Intersection};

\end{tikzpicture}
\end{center}

Without loss of generality, assume $\tau_p \leq \tau_q$ (otherwise exchange roles of $p$ and $q$).
Select $\beta \in \mathbb{Q} \cap (0,1)$ as in Lemma \ref{Lemma4-Rigidity} such that $\dfrac{a\beta+b}{a_i} \notin \mathbb{Q}$ for all $a_i$.
For $m,n \in \mathbb{N}$, define $L_1 = L_4 = m\tau_p$ and $L_2 = L_3 = n\tau_q$ with $m = \lfloor \beta n \rfloor$. Then $L_1 = L_4 \leq L_2 = L_3$, and 
\[
\lim_{n\to\infty}\frac{L_1}{L_2} = \beta\frac{\tau_p}{\tau_q}  :=2\alpha \in (0,1).
    \]
For sufficiently large $n$, we have $L_1 \geq \alpha L_2$, so setting $L_0 := \min \{L_1,L_2\} = L_1$ yields $L_0 \geq \alpha L_i$ for $i=1,2,3,4$. Construct the finite sequence:
\[
Q = \{\varphi^{-L_1}(y), \ldots, \varphi^{-1}(y), y, \varphi(y), \ldots, \varphi^{L_2 - 1}(y), \varphi^{-L_3}(x), \ldots, \varphi^{-1}(x), x, \varphi(x), \ldots, \varphi^{L_4 - 1}(x)\},
\]
which forms a $\delta$-pseudo-orbit for appropriate $\delta>0$.

From (\ref{Eq1:Rig}) and hyperbolicity, there exists $K > 0$ such that for all $n \geq 0$:
\begin{align}\label{eq2:Rigi}
d(\varphi^n(x), \varphi^n(p)) &\leq K \lambda^n d(x,p), & d(\varphi^{-n}(x), \varphi^{-n}(q)) &\leq K \lambda^n d(x,q), \\
d(\varphi^n(y), \varphi^n(q)) &\leq K \lambda^n d(y,q), & d(\varphi^{-n}(y), \varphi^{-n}(p)) &\leq K \lambda^n d(y,p). \nonumber
\end{align}
Note that $d(\varphi^{L_2}(q), \varphi^{-L_3}(q)) = 0$ since both $L_2$ and $L_3$ are multiples of $\tau_q$.

We estimate:
\begin{align*}
d(\varphi(\varphi^{L_2-1}(y)), \varphi^{-L_3}(x)) 
&\leq d(\varphi^{L_2}(y), \varphi^{L_2}(q)) + d(\varphi^{L_2}(q), \varphi^{-L_3}(q)) + d(\varphi^{-L_3}(q), \varphi^{-L_3}(x)) \\
&\leq K\lambda^{L_2}d(y,q) + K\lambda^{L_3}d(x,q).
\end{align*}
Similarly,
\[
d(\varphi(\varphi^{L_4-1}(x)), \varphi^{-L_1}(y)) \leq K\lambda^{L_4}d(x,p) + K\lambda^{L_1}d(y,p).
\]
Let $\delta_0 = \max\{d(x,p), d(y,q), d(x,q), d(y,p)\}$. With $L_0 := \min L_i$, $Q$ is a periodic $\delta$-pseudo-orbit with $\delta := 2K\lambda^{L_0}\delta_0$ and period $L = \sum_{i=1}^{4}L_i$.

By the Lipschitz pseudo-orbit shadowing Lemma for basic sets, there exists an $L$-periodic point $z$ that $\mu\delta$-shadows $Q$. Assume without loss of generality that $z$ begins shadowing at $y$:
\begin{align*}
d(\varphi^{-i}(z), \varphi^{-i}(y)) &\leq \mu\delta, & i=1,\dots, L_1, \\
d(\varphi^{-L_1-j}(z), \varphi^{L_4-j}(x)) &\leq \mu\delta, & j=1,\dots, L_4, \\
d(\varphi^{-(L_1+L_4)-j}(z), \varphi^{-j}(x)) &\leq \mu\delta, & j=1,\dots, L_3, \\
d(\varphi^{-(L_1+L_4+L_3)-j}(z), \varphi^{L_2-j}(y)) &\leq \mu\delta, & j=1,\dots,L_2.
\end{align*}

Now employ the H\"older continuity of $f$. Let $C > 0$ and $\theta \in (0,1)$ be the H\"older constants. From (\ref{eq2:Rigi}) and the shadowing estimates:
\begin{align*}
\sum_{i=1}^{L_1}|f(\varphi^{-i}(z)) - f(\varphi^{-i}(y))| 
&\leq C L_1 (\mu\delta)^\theta \\
&= C L_1 \mu^\theta (2K\lambda^{L_0}\delta_0)^{\theta} \\
&\leq C (2K\mu\delta_0)^{\theta} L_1\lambda^{\theta L_1}\\
&\leq C (2K\mu\delta_0)^{\theta} L_1\lambda^{\theta \alpha L_1}.
\end{align*}
Similarly,
\[
\sum_{i=1}^{L_2}|f(\varphi^{L_2-i}(z)) - f(\varphi^{L_2-i}(y))| \leq C (2K\mu\delta_0)^{\theta} L_2\lambda^{\theta \alpha L_2}.
\]
Since $\lambda \in (0,1)$, we have $\lim_{L_i \to \infty} L_i\lambda^{\theta \alpha L_i} = 0$ for each $i$.

Choose $\epsilon < \frac{\min a_i}{4k}$ (this choice of $\epsilon$ is relate with  Lemma \ref{Lemma3-Rigidity}). For sufficiently large $n$ and $m$ (ensuring $L_1 = \lfloor\beta n \rfloor\tau_p$ and $L_2=n\tau_q$ are large), we obtain:
\[
\sum_{i=1}^{L_1}\left(f(\varphi^{-i}(z)) - f(\varphi^{-i}(y))\right), \ 
\sum_{i=1}^{L_2}\left(f(\varphi^{L_2-i}(z)) - f(\varphi^{L_2-i}(y))\right) \in \left(-\frac{\epsilon}{9}, \frac{\epsilon}{9}\right).
\]
Summing yields:
\begin{equation}\label{eq3:Rigi}
    \sum_{i=-L_1}^{L_2-1}\left(f(\varphi^{i}(z)) - f(\varphi^{i}(y))\right) \in \left(-\frac{2\epsilon}{9}, \frac{2\epsilon}{9}\right).
\end{equation}
Similarly,
\begin{equation}\label{eq4:Rigi}
    \sum_{i=L_2}^{L-L_1}f(\varphi^{i}(z)) - \sum_{i=-L_3}^{L_4-1}f(\varphi^{i}(x)) \in \left(-\frac{2\epsilon}{9}, \frac{2\epsilon}{9}\right).
\end{equation}
Combining these estimates:
\begin{equation}\label{eq5:Rigi}
    \sum_{i=0}^{L-1}f(\varphi^{i}(z)) - \left(\sum_{j=1}^{L_1} f(\varphi^{-i}(y)) + \sum_{j=0}^{L_2 - 1} f(\varphi^{i}(y)) + \sum_{j=1}^{L_3} f(\varphi^{-i}(x)) + \sum_{j=0}^{L_4 - 1} f(\varphi^{i}(x))\right) \in \left(-\frac{4\epsilon}{9}, \frac{4\epsilon}{9}\right).
\end{equation}

\noindent We now compare these expressions with the Birkhoff sums at $p$ and $q$. Estimate:
\[
|f(\varphi^{-i \tau_p - j}(y)) - f(\varphi^{-j}(p))| \leq C K^\theta \lambda^{\theta (j+i \tau_p)} (d(y, p))^{\theta}.
\]
Since $\lambda \in (0,1)$, the series
\[
\sum_{j=0}^{\infty}\left(f(\varphi^{-i \tau_p - j}(y)) - f(\varphi^{-j}(p))\right)
\]
converges absolutely for each $i=1,\dots, \tau_p$. Define:
\begin{align*}
H_1 &:= \sum_{i=1}^{\tau_p}\sum_{j=0}^{\infty}\left(f(\varphi^{-i \tau_p - j}(y)) - f(\varphi^{-j}(p))\right), \\
H_2 &:= \sum_{i=0}^{\tau_q-1} \sum_{j=0}^{\infty} \left(f(\varphi^{i \tau_q + j}(y)) - f(\varphi^{j}(q))\right), \\
H_3 &:= \sum_{i=1}^{\tau_q} \sum_{j=0}^{\infty} \left(f(\varphi^{-i \tau_q - j}(x)) - f(\varphi^{-j}(q))\right), \\
H_4 &:= \sum_{i=0}^{\tau_p-1}\sum_{j=0}^{\infty}\left(f(\varphi^{i \tau_p + j}(x)) - f(\varphi^{j}(p))\right).
\end{align*}
For large $n$ and $m$:
\begin{align*}
\sum_{j=1}^{L_1} f(\varphi^{-i}(y)) - m S(f,\varphi,p) &\in \left(H_1 - \frac{\varepsilon}{9}, H_1 + \frac{\varepsilon}{9}\right), \\
\sum_{j=0}^{L_2 - 1} f(\varphi^{i}(y)) - n S(f,\varphi,q) &\in \left(H_2 - \frac{\varepsilon}{9}, H_2 + \frac{\varepsilon}{9}\right), \\
\sum_{j=1}^{L_3} f(\varphi^{-i}(x)) - n S(f,\varphi,q) &\in \left(H_3 - \frac{\varepsilon}{9}, H_3 + \frac{\varepsilon}{9}\right), \\
\sum_{j=0}^{L_4 - 1} f(\varphi^{i}(x)) - m S(f,\varphi,p) &\in \left(H_4 - \frac{\varepsilon}{9}, H_4 + \frac{\varepsilon}{9}\right).
\end{align*}
Combining with (\ref{eq5:Rigi}):
\begin{equation}\label{eq6:Rigi}
    \sum_{i=0}^{L-1}f(\varphi^{i}(z)) - \left(2mS(f,\varphi,p) + 2nS(f,\varphi,q)\right) \in \left(H - \frac{8\epsilon}{9}, H + \frac{8\epsilon}{9}\right),
\end{equation}
where $H = H_1 + H_2 + H_3 + H_4$.

The periodic point $z$ depends on $n$ and $m$ (through $m = \lfloor\beta n \rfloor$). By hypothesis $\sum_{i=0}^{L-1}f(\varphi^{i}(z)) \in \bigcup_{i=1}^{k}a_i\mathbb{N}$. Consider $m_0$ and $n_0$ large enough (sufficiently large such that $\frac{L_1}{L_2} > \alpha$) and such that (\ref{eq6:Rigi}) holds for all $m\geq m_0$ and $n\geq n_0$. Consider $z(n)$ the periodic point associated to $m=\lfloor\beta n \rfloor$ and $n$, for $n\geq n_0$. Define:
\[
A_i := \{n \geq n_0 : S(f,\varphi, z(n)) \in a_i\mathbb{N}\}.
\]
Then $\bigcup_{i=1}^{k}A_i = \mathbb{N} \cap [n_0, \infty)$. From Lemma \ref{LEMMA3NEW-RIG} there is $A_{i_0}$ with $d^*(A_{i_0})\geq\frac{1}{k}$ 
For $n_j \in A_{i_0}$, we have $S(f,\varphi, z(n_j)) \in a_{i_0}\mathbb{N}$. Define $G(n_j) = S(f,\varphi, z(n_j))/a_{i_0}$. By Lemma \ref{Lemma3-Rigidity}, with $a = 2S(f,\varphi,p)$ and $b = 2S(f,\varphi,q)$ satisfying $a/b \notin \mathbb{Q}$, the set
\[
S := \{a_{i_0}G(n) - (am + bn) : m = \lfloor\beta n \rfloor, n = n_j\}
\]
satisfies $S \not\subset [H - \rho, H + \rho]$ for all $\rho < a_{i_0}/4k$.

Since $8\epsilon/9 < \epsilon < \min a_i/4k \leq a_{i_0}/4k$, this contradicts (\ref{eq6:Rigi}).

We conclude that for all $p \in \text{Per}(\varphi)$, $S(f,\varphi, p) \in \bigcup_{i=1}^{k}a_i\mathbb{N}$ with rational ratios $a_i/a_j$. The result now follows from Case 1.
\end{proof}


\subsection{Proof of Theorem \ref{mainT2}}
Before to prove this Theorem we must to understand a little more the proof of the Theorem \ref{Rigidity}, more specifically, we must to estimative $H_i$, $i=1,\dots 4$.\\
In fact, note that $y\in W^u(x)$ and $f$ is H\"older constant $C>0$ and $0<\theta<1$.
\begin{align*}
|H_1| &\leq \sum_{i=1}^{\tau_p}\sum_{j=0}^{\infty}|f(\varphi^{-i \tau_p - j}(y)) - f(\varphi^{-j}(p))|, 
\leq \sum_{i=1}^{\tau_p}\sum_{j=0}^{\infty} C\cdot (d(\varphi^{-i \tau_p - j}(y), \varphi^{-j}(p)))^{\theta}\\
&\leq \sum_{i=1}^{\tau_p}\sum_{j=0}^{\infty} C\cdot K^\theta \cdot \lambda^{(i\tau_p+j)\theta}\cdot d(y,p)^{\theta}
= \sum_{i=1}^{\tau_p}C\cdot K^\theta \cdot \lambda^{i\tau_p} \dfrac{1}{1-\lambda^\theta}\cdot d(y,p)^{\theta}\\
&=\dfrac{C\cdot K^\theta}{1-\lambda^\theta}\cdot \dfrac{\lambda^{\theta\tau_p}(1-\lambda^{\theta{\tau^2_p}})}{1-\lambda^{\theta\tau_p}} d(y,p)^{\theta}: = F_1(\theta, p)d(y,p)^{\theta}
\end{align*}
Similarly, we can see that 
\begin{align*}
    |H_2| &=\dfrac{C\cdot K^\theta}{1-\lambda^\theta}\cdot \dfrac{1-\lambda^{\theta\tau_q^2}}{1-\lambda^{\theta\tau_q}}\cdot d(y,q)^{\theta}: = F_1(\theta, q)d(y,q)^{\theta}\\
    |H_3|&\leq \dfrac{C\cdot K^\theta}{1-\lambda^\theta}\cdot \dfrac{\lambda^{\theta\tau_q}(1-\lambda^{\theta\tau_q^2})}{1-\lambda^{\theta\tau_q}}d(x,q)^{\theta}:= F_3(\theta, q)d(x,q)^{\theta}\\
    |H_4| &\leq \dfrac{C\cdot K^\theta}{1-\lambda^\theta}\cdot \dfrac{1-\lambda^{\theta\tau_p^2}}{1-\lambda^{\theta\tau_p}} d(x,p)^{\theta}:=F_4(\theta, p)d(x,p)^{\theta}.
\end{align*}
Since $ \dfrac{1-\lambda^{\theta\tau_p^2}}{1-\lambda^{\theta\tau_p}}$ is bounded in $\tau_p$, then there is a constant $F>0$, such that $F_i\leq \Gamma$, $i=1,2,3,4$.
In particular, 

\begin{equation}\label{EQ1-Bracket}
    |H:=H_1+H_2+H_3+H_4|\leq \Gamma \cdot \max \{d(x,p)^{\theta}, d(x,q)^{\theta}, d(y,p)^{\theta}, d(y,q)^{\theta}\}.
\end{equation}

\begin{remark}
    From the point of view of bracket, for $p,q$ close enough, \emph{(\ref{EQ1-Bracket})} can be write as 
    \begin{equation}\label{EQ2-Bracket}
        |H|\leq \Gamma \cdot \max\left\{d(p, [p,q])^{\theta}, d(p, [q,p])^{\theta}, d(q, [p,q])^{\theta}, d(q, [q,p])^{\theta}\right\}.
    \end{equation}
\end{remark}
\begin{proof}[\emph{\textbf{Proof of Theorem \ref{mainT2}}}]
Assume the case $\mathcal{B}(f,\Lambda)\subset {\mathbb{R}}^{+}_{\geq 0}$; the other case is analogous. \\ If $\mathcal{B}(f,\Lambda)$ is dense in $[0,\infty)$, then of course, $f$ is not cohomologous to zero and $\inf \mathcal{B}(f,\Lambda)=0$. Let us prove the converse. \\
Let $A>0$ and $\varepsilon >0$ be small. Since $f$ is not cohomologous to zero and $\inf \mathcal{B}(f,\Lambda)=0$, there exists a sequence $p_n\in \text{Per}(\varphi)$ such that $S(f,\varphi, p_n)\to 0$.  

Let $\delta>0$ be as in Corollary \ref{Coro- Bracket}. From (\ref{EQ2-Bracket}), given $m_0,n_0\in \mathbb{N}$, there exist periodic points $p_n:=p$ and $q_n:=q$ with $d(p,q)<\delta$ and the following properties:

\begin{itemize}
    \item[(a)] $|H|\leq\Gamma\cdot 
\max\left\{d(p, [p,q])^\theta,\ d(p, [q,p])^\theta,\ d(q, [p,q])^\theta,\ d(q, [q,p])^\theta\right\} < \varepsilon
$;
    \item[(b)] $2m_0S(f,\varphi, p)<\dfrac{\varepsilon}{5}$ and $\Big\lfloor \dfrac{A}{S(f,\varphi, q)}\Big\rfloor>2n_0$, with $S(f,\varphi, q)<\dfrac{\varepsilon}{10}$.
\end{itemize}

Following the same argument as in the proof of Theorem \ref{Rigidity}, we obtain, similarly to (\ref{eq6:Rigi}), that there exists $z\in \text{Per}(\varphi)$ such that 

\[
\sum_{i=0}^{L-1}f(\varphi^{i}(z)) - \left(2m_0S(f,\varphi,p) + 2\Big\lfloor \dfrac{A}{2S(f,\varphi, q)}\Big\rfloor S(f,\varphi,q)\right) \in \left(H - \frac{8\epsilon}{9}, H + \frac{8\epsilon}{9}\right).
\]

From item (b) we have 
\[
\left(2m_0S(f,\varphi,p) + 2\Big\lfloor \dfrac{A}{2S(f,\varphi, q)}\Big\rfloor S(f,\varphi,q)\right)\in \Big[A+\frac{\epsilon}{5}, A+\frac{2\epsilon}{5}\Big).
\]
Thus, using item (a),
\[
\sum_{i=0}^{L-1}f(\varphi^{i}(z)) \in \Big[A+\frac{\epsilon}{5}, A+\frac{2\epsilon}{5}\Big) + \Big (-\frac{17\epsilon}{9}, \frac{17\epsilon}{9} \Big)\subset \Big( A- 2\epsilon, A+ 2\epsilon\Big),
\]
which establishes the desired density property.

\end{proof}

\section{Proof of Main Results for Flows}\label{Section for Flows}

The proof of this theorems for flows require the construction of a special hyperbolic set $\Lambda$ for $\phi^t$ and the construction of an appropriate Poincar\'e map associated with $\Lambda$, which will be done in the following three lemmas (cf. \cite{Romana1} for a good lecture about reduction to Poincar\'e maps). 

\subsection{Reduction to Poincar\'e Maps}\label{reduction via Poincare Maps}
\subsubsection{Existence of Flow Box Cover with Disjoint Bases}
\begin{lem}\label{Good Cross Section 1}
Let $\Lambda$ be a compact, locally maximal hyperbolic invariant set for a smooth flow $\varphi_t$ on a manifold $M$. Then there exists a finite collection of flow boxes $\{B_i\}_{i=1}^k$ covering $\Lambda$ such that their bases are disjoint.
\end{lem}

\begin{proof}
We proceed in several steps:

\paragraph{Step 1: Definitions and setup.}
A \textbf{flow box} is a set $B = \bigcup_{|t|<\epsilon} \varphi_t(D)$, where $D \subset M$ is an embedded $(n-1)$-disk transverse to the flow, called the \textbf{base}. Two flow boxes have \textbf{disjoint bases} if their bases $D_i, D_j$ are disjoint as subsets of $M$.

The hyperbolic splitting on $\Lambda$ is:
\[
T_\Lambda M = E^s \oplus E^u \oplus \mathbb{R} X,
\]
where $X$ is the generator of $\varphi_t$.

By \textbf{local maximality}, there exists an open neighborhood $U$ of $\Lambda$ such that
\[
\Lambda = \bigcap_{t \in \mathbb{R}} \varphi_t(U).
\]

\paragraph{Step 2: Local product structure.}
For a hyperbolic set, there exists $\delta>0$ such that for all $x,y \in \Lambda$ with $d(x,y) < \delta$, $W^s_\delta(x)$ and $W^u_\delta(y)$ intersect in exactly one point $[x,y] \in \Lambda$, and the intersection is transverse. Here $W^s_\delta(x)$ and $W^u_\delta(y)$ are local stable/unstable manifolds of size $\delta$.

This implies that for each $x \in \Lambda$, there is a neighborhood $N_x$ in $\Lambda$ homeomorphic to $(-\epsilon,\epsilon) \times D^s \times D^u$ (where $D^s, D^u$ are disks in $E^s, E^u$ respectively), with the flow acting as translation in the first coordinate.

\paragraph{Step 3: Choosing a finite cover of $\Lambda$ by flow boxes.}
Pick $r>0$ small enough so that for each $x \in \Lambda$,
\[
\Lambda \cap B(x,2r) \subset N_x
\]
in the local product chart.

Let $\{x_1,\dots,x_m\} \subset \Lambda$ be such that $\Lambda \subset \bigcup_{i=1}^m B(x_i, r)$.

In each $B(x_i, 2r)$, the set $\Lambda$ is a union of local orbits $\varphi_t(D_i)$, where $D_i$ is a local cross-section through $x_i$ given by $\{0\} \times D^s \times D^u$ in the local coordinates.

Let $S_i = D_i \cap \Lambda$. This $S_i$ is a Cantor set in general, but we can thicken it slightly in $M$ to a disk $\widetilde{D}_i$ transverse to the flow, such that $\widetilde{D}_i \subset B(x_i, 2r)$ and $S_i \subset \operatorname{int}(\widetilde{D}_i)$.

Define $B_i = \varphi_{(-\eta,\eta)}(\widetilde{D}_i)$ for small $\eta>0$ (which depend on $r$) so that $B_i \subset B(x_i, 2r)$.

Then $\{B_i\}$ covers $\Lambda$ because each orbit through $\Lambda$ passes within $r$ of some $x_i$, hence hits $\widetilde{D}_i$ in time $<\eta$ if $\eta$ is small enough relative to the flow speed and $r$.

\paragraph{Step 4: Making the bases disjoint.}
We have $\widetilde{D}_i$ possibly overlapping for different $i$. We want new bases $D_i'$ disjoint.

We use the \textbf{Markov partition} construction for flows (Bowen, 1970--73): There exists a finite collection of disjoint sections $\Sigma_1, \dots, \Sigma_k$ (each a union of disks) and a partition of $\Lambda \cap \bigcup \Sigma_j$ into rectangles $R_\alpha \subset \Sigma_j$ such that:
\begin{enumerate}
\item Each orbit of $\Lambda$ meets $\bigcup_j \Sigma_j$.
\item The first return map to $\bigcup_j \Sigma_j$ is a Markov map.
\item The $R_\alpha$ have disjoint interiors within each $\Sigma_j$, and if $R_\alpha \subset \Sigma_j$, $R_\beta \subset \Sigma_\ell$ with $j \neq \ell$, then $R_\alpha$ and $R_\beta$ are disjoint as subsets of $M$ because $\Sigma_j \cap \Sigma_\ell = \varnothing$ for $j \neq \ell$.
\end{enumerate}

Let $D_\alpha$ be a small open neighborhood of $R_\alpha$ within $\Sigma_j$, such that:
\begin{itemize}
\item $D_\alpha \subset \Sigma_j$,
\item $D_\alpha \cap D_\beta = \varnothing$ for $\alpha \neq \beta$,
\item $\Lambda \cap \Sigma_j \subset \bigcup_\alpha D_\alpha$.
\end{itemize}

Now define flow boxes $B_\alpha = \varphi_{(-\epsilon,\epsilon)}(D_\alpha)$ for small $\epsilon>0$ less than the return time to $\bigcup \Sigma_j$.

\paragraph{Step 5: Verifying the cover.}
Any $z \in \Lambda$ has its orbit hitting some $R_\alpha$ in $\bigcup \Sigma_j$ by property (1) of Markov partitions. Thus $z = \varphi_t(y)$ for some $y \in R_\alpha \subset D_\alpha$ and $|t|<\epsilon$ if we choose $\epsilon$ small enough (less than the minimum return time to $\bigcup \Sigma_j$ for points in $R_\alpha$).

Hence $z \in B_\alpha$.

The bases $D_\alpha$ are disjoint by construction.

\paragraph{Step 6: Conclusion.}
We have constructed a finite set of flow boxes $B_\alpha$ covering $\Lambda$, with disjoint bases $D_\alpha$.
\end{proof}

\subsubsection{Poincar\'e Map on Disjoint Base Sections Preserves Basic Set Structure}

\begin{lem}\label{poincare}
Let $\Lambda$ be a compact, locally maximal hyperbolic basic set for a smooth flow $\varphi_t$ on a manifold $M$. Let $\mathcal{D} = \bigcup_{i=1}^k D_i$ be a finite collection of disjoint transverse sections such that $\Lambda \subset \bigcup_{i=1}^k \varphi_{(-\epsilon,\epsilon)}(D_i)$ for some $\epsilon > 0$. Let $\mathcal{P}: \mathcal{D}  \to \mathcal{D}$ be the Poincar\'e first return map. Then $\Lambda \cap \mathcal{D}$ is a basic set for $\mathcal{P}$.
\end{lem}

\begin{proof}
We verify the three properties of a basic set for the discrete dynamical system $(\mathcal{P}, \Lambda \cap \mathcal{D})$.

\paragraph{Step 1: Hyperbolicity of $\mathcal{P}$ on $\Lambda \cap \mathcal{D}$.}

Since $\Lambda$ is hyperbolic for the flow, we have the continuous splitting:
\[
T_x M = E^s(x) \oplus E^u(x) \oplus \mathbb{R} X(x) \quad \text{for all } x \in \Lambda,
\]
where $X$ is the generator of $\varphi_t$, with uniform contraction/expansion rates.

For $x \in \mathcal{D} \cap \Lambda$, the tangent space $T_x \mathcal{D}$ is transverse to $X(x)$. The hyperbolicity of $\mathcal{P}$ follows from:
\begin{itemize}
\item The stable manifold $W^s(x) \cap \mathcal{D}$ (for the flow) provides the stable manifold for $\mathcal{P}$, with $E^s_\mathcal{P}(x) = E^s(x) \cap T_x \mathcal{D}$.
\item The unstable manifold $W^u(x) \cap \mathcal{D}$ provides the unstable manifold for $\mathcal{P}$, with $E^u_\mathcal{P}(x) = E^u(x) \cap T_x \mathcal{D}$.
\item The splitting $T_x \mathcal{D} = E^s_\mathcal{P}(x) \oplus E^u_\mathcal{P}(x)$ is invariant under $D\mathcal{P}$ since the flow invariant manifolds are preserved by $\varphi_t$ and hence by $\mathcal{P}$.
\item Uniform contraction/expansion under $D\mathcal{P}$ follows from the flow's uniform rates and the fact that the return time $\tau(x)$ for $\mathcal{P}$ is bounded away from $0$ and $\infty$ by compactness of $\Lambda$ and transversality of $\mathcal{D}$.
\end{itemize}
Thus, $\Lambda \cap \mathcal{D}$ is uniformly hyperbolic for $\mathcal{P}$.

\paragraph{Step 2: Local maximality of $\Lambda \cap \mathcal{D}$ for $\mathcal{P}$.}

Since $\Lambda$ is locally maximal for the flow, there exists an open neighborhood $U$ of $\Lambda$ such that:
\[
\Lambda = \bigcap_{t \in \mathbb{R}} \varphi_t(U).
\]
Define $V = U \cap \mathcal{D}$, which is an open neighborhood of $\Lambda \cap \mathcal{D}$ in $\mathcal{D}$.

We claim:
\[
\Lambda \cap \mathcal{D} = \bigcap_{n \in \mathbb{Z}} \mathcal{P}^n(V).
\]

\emph{Proof of claim:}
\begin{itemize}
\item[($\subseteq$)] If $x \in \Lambda \cap \mathcal{D}$, then its entire flow orbit lies in $\Lambda \subset U$, so all its $\mathcal{P}$-iterates remain in $V$. Thus $x \in \bigcap_{n \in \mathbb{Z}} \mathcal{P}^n(V)$.

\item[($\supseteq$)] If $x \in \bigcap_{n \in \mathbb{Z}} \mathcal{P}^n(V)$, then there exists a full $\mathcal{P}$-orbit $(x_n)$ in $V$ with $x_0 = x$. By the construction of $\mathcal{P}$, there exists a flow orbit $\gamma$ through $x$ that intersects $\mathcal{D}$ at each $x_n$. Since each $x_n \in V \subset U$, the flow orbit $\gamma$ remains in $U$ for all time. By local maximality of $\Lambda$ for the flow, $\gamma \subset \Lambda$. Hence $x \in \Lambda \cap \mathcal{D}$.
\end{itemize}
This establishes local maximality of $\Lambda \cap \mathcal{D}$ for $\mathcal{P}$.

\paragraph{Step 3: Transitivity of $\Lambda \cap \mathcal{D}$ for $\mathcal{P}$.}

Since $\Lambda$ is a basic set for the flow, it is transitive: there exists $z \in \Lambda$ whose flow orbit $\{\varphi_t(z) : t \in \mathbb{R}\}$ is dense in $\Lambda$.

Let $\{t_n\}$ be the sequence of times (bi-infinite) such that $\varphi_{t_n}(z) \in \mathcal{D}$, with $t_0 = 0$ (after possible time shift). This sequence is infinite in both directions because the flow orbit is dense and $\mathcal{D}$ is a union of transverse sections covering $\Lambda$ in the flow-box sense.

Let $x_n = \varphi_{t_n}(z) \in \Lambda \cap \mathcal{D}$. Then $x_{n+1} = \mathcal{P}(x_n)$. The density of the flow orbit in $\Lambda$ implies that $\{x_n\}$ is dense in $\Lambda \cap \mathcal{D}$ under the $\mathcal{P}$-iteration. Thus $\Lambda \cap \mathcal{D}$ is transitive for $\mathcal{P}$.

\paragraph{Conclusion.}
We have shown that $\Lambda \cap \mathcal{D}$ is:
\begin{enumerate}
\item uniformly hyperbolic for $\mathcal{P}$,
\item locally maximal for $\mathcal{P}$,
\item transitive for $\mathcal{P}$.
\end{enumerate}
Therefore, $\Lambda \cap \mathcal{D}$ is a basic set for the Poincar\'e map $\mathcal{P}$.

\end{proof}

For transitive Anosov flow, we obtain:

\begin{lem}\label{Good Cross Section} Let $\phi^t\colon N \to N$ be a transitive Anosov flow and $\theta, \omega \in \text{Per}(\phi^t)$, then 
\begin{enumerate}
    \item[\emph{1}.] There is a basic set  $\Lambda(\theta,\omega)\subsetneq N$ for $\phi^t$ which contains $\theta$ and $\omega$.
    \item[\emph{2}.] There are  $\gamma>0$ and a finite number of cross section $\Sigma_i$, $i=1,\dots, k$  such that 
$$\Lambda(\theta, \omega) \subset \bigcup_{i=1}^k \phi^{(-\delta, \delta)}(\emph{int}\,\Sigma_i),$$
with $\Sigma_i\cap \Sigma_j= \emptyset$.
\end{enumerate}
\end{lem}
\begin{proof}
    From Lemma \ref{Good Cross Section 1}, it suffices to prove item 1. Since \(\phi^t\) is transitive, the spectral decomposition theorem (see \cite[Section 18.3]{Katok - Hasselblatt}) implies that \(\theta\) and \(\omega\) are related, meaning
\[
W^{cs}(\theta) \pitchfork W^u(\omega) \quad \text{and} \quad W^{cu}(\theta) \pitchfork W^s(\omega).
\]
Hence, the construction of \(\Lambda(\theta,\omega)\) follows the same argument as in \cite[Theorem 6.5.5]{Katok - Hasselblatt}.
\end{proof}
\subsection{Proof of Theorem \ref{MainT2}}

\begin{proof}[\emph{\textbf{Proof of of Theorem \ref{MainT2}}}]
First for basic sets, assume \( f \) is a dispersed and non-arithmetic function, meaning there exist \( \theta, \omega \in \Lambda \) such that
\[
\oint_{\mathcal{O}(\theta)} f < 0 < \oint_{\mathcal{O}(\omega)} f .
\]
Let \( \mathcal{P} : \mathcal{D} \to \mathcal{D} \) be the Poincar\'e first return map provided by Lemma \ref{Good Cross Section 1} and Lemma \ref{poincare}. For each \( x \in \mathcal{D} \), let \( \mathcal{P}(x) = \phi^{t(x)}(x) \). We then define the \( C^1 \) function \( \tilde{f} : \mathcal{D} \to \mathbb{R} \) by
\[
\tilde{f}(x) = \int_{0}^{t(x)} f(\phi^s(x)) \, ds .
\]
If \( x \in \operatorname{Per}(\phi^t) \cap \mathcal{D} \), then \( x \) is periodic for \( \mathcal{P} \) with period \( \tau_\mathcal{P}(x) \), and
\[
\oint_{\mathcal{O}(x)} f = \sum_{i=0}^{\tau_{\mathcal{P}}(x)-1} \tilde{f}(\mathcal{P}^i(x)) = S(\tilde{f}, \mathcal{P}, x).
\]
In other words, \( \mathcal{B}(f, \phi^t|_{\Lambda}) = \mathcal{B}(\tilde{f}, \mathcal{P}, \mathcal{D} \cap \Lambda) \).
Now, for a general \( x \in \Lambda \), let
\[
t_1(x) = \min\{ t \geq 0 : \phi^t(x) \in \mathcal{D} \},
\]
and set \( x_1 = \phi^{t_1(x)}(x) \in \mathcal{D} \). Applying this to \( \theta \) and \( \omega \) gives
\[
S(\tilde{f}, \mathcal{P}, \theta_1) < 0 < S(\tilde{f}, \mathcal{P}, \omega_1).
\]
Therefore, by Theorem \ref{Shaobo}, we have 
\(
\overline{\mathcal{B}(\tilde{f}, \mathcal{P}, \mathcal{D} \cap \Lambda)} = \mathbb{R},
\)
and consequently
\[
\overline{\mathcal{B}(f, \phi^t|_{\Lambda})} = \mathbb{R},
\]
as desired. \\
If $\phi^t$ Anosov, it is Axiom A and the result follow for the first case. For the case of transitive Anosov flows, If $f$ is dispersed and non-arithmetic, then there exist $ \theta, \omega \in \operatorname{Per}(\phi^t)$ such that
\[
\oint_{\mathcal{O}(\theta)} f < 0 < \oint_{\mathcal{O}(\omega)} f .
\]
By Lemma \ref{Good Cross Section}, there exists a basic set $\Lambda(\theta, \omega) \subset N$ containing both $\theta$ and $\omega$. The desired conclusion then follows directly from the fist part of the theorem.

\end{proof}

When $f$ is not dispersed, we have the following result, in the same spirit of Theorem \ref{mainT}.
we obtain 

\begin{coro} \label{C3-GS}
   If $\phi^t$ is an Anosov flow and $f$ is a H\"older continuous function, then the following are equivalent: 
\begin{itemize}
    \item[\emph{(i)}] $\mathcal{B}(f,\phi^t)$ is a bounded set.
    
    \item[\emph{(ii)}] $f$ is cohomologous to $0$. 

    \item[\emph{(iii)}] $\mathcal{B}(f,\phi^t)=\{0\}$. 
\end{itemize}
\end{coro}

\subsection{Some applications to geodesic flow}\label{geodesic flow}
Since the geodesic flow preserves the \emph{Liouville measure}, it is transitive (see \cite{Pa}). So, as a consequence of Theorem \ref{MainT2}, we obtain the Corollary \ref{Anosov geodesic flow}.

When $f$ is the Ricci Curvature (cf. \cite{doCarmo} for more details), then we obtain the following characterization (compare with \cite{IR2, Nina2, Romana}):
\begin{teo}\label{thm:curvature-rigidity}
Let \((M,g)\) be an Anosov manifold whose sectional curvature satisfies \(K \geq -c^2\) for some \(c \geq 0\). Then \(\mathcal{B}(\mathrm{Ric}+c^2, \phi^t)\) is bounded if and only if \((M,g)\) has constant sectional curvature \(-c^2\).
\end{teo}

\begin{proof}
From Corollary \ref{C3-GS}, if \(\mathcal{B}(\mathrm{Ric}+c^2,\phi^t)\) is bounded, 
    then \(\mathcal{B}(\mathrm{Ric}+c^2,\phi^t)=\{0\}\). Consequently, for every 
    periodic orbit we have
    \[
    \frac{1}{\tau_\theta}\int_{0}^{\tau_\theta}\mathrm{Ric}(\phi^s(\theta))\,ds = -c^2,
    \qquad \text{for all } \theta\in \mathrm{Per}(\phi^t).
    \]
    Since \(K_M \geq -c^2\), it follows that \(\mathrm{Ric}(\cdot) \geq -c^2\) pointwise.
    Combined with the identity above, this forces
    \[
    \mathrm{Ric}|_{\mathcal{O}(\theta)} \equiv -c^2 \quad \text{for every } \theta\in \mathrm{Per}(\phi^t).
    \]
    Because periodic orbits are dense in \(SM\), we conclude that \(\mathrm{Ric}(\cdot) \equiv -c^2\) 
    everywhere, and therefore \(K_M \equiv -c^2\).

\end{proof}

\begin{remark}\label{No bounded}
Recall that the unstable Lyapunov exponent for the geodesic flow is defined by
\[
\chi^{+}(\theta) := \lim_{t \to \infty} \frac{1}{t} \log \bigl\lvert \det \bigl( D\phi^t_{\theta}|_{E^u_\theta} \bigr) \bigr\rvert,
\]
where \(E^u\) denotes the unstable bundle \emph{(see \cite{IR2,IR1, Romana})}. It satisfies the inequality
\[
\chi^{+}(\theta) \leq (n-1)\sqrt{-\lim_{t\to\infty}\frac{1}{t}\int_{0}^{t} \operatorname{Ric}\bigl(\phi^s_M(\theta)\bigr)\,ds } .
\]
If \(\theta \in \operatorname{Per}(\phi^t)\) is a periodic point with period \(\tau_\theta\), this estimate simplifies to
\[
\chi^{+}(\theta) \leq (n-1)\sqrt{-\frac{1}{\tau_\theta}\int_{0}^{\tau_\theta} \operatorname{Ric}\bigl(\phi^s_M(\theta)\bigr)\,ds } .
\]
In the Anosov case, there exists a contraction constant \(\lambda \in (0,1)\) such that
\[
-\log \lambda \leq \chi^{+}(\theta) \quad \text{for all } \theta \in \operatorname{Per}(\phi^t).
\]
Consequently, \(\mathcal{B}(\operatorname{Ric}, \phi^t)\) is unbounded, and 
\[
\mathcal{B}(\operatorname{Ric}, \phi^t) \subset \Bigl(-\infty, -\dfrac{(\log\lambda)^2}{(n-1)^2}\Bigr].\]
\end{remark}

Recall that, for Anosov geodesic flows, the set of closed orbits is dense in \( SM \); consequently, closed geodesics form a dense subset of \( M \). Given a real function \( f \colon M \to \mathbb{R} \), the Birkhoff spectrum associated to \( M \) and \( f \) is defined by:
\[
\mathcal{B}_{M}(f) := \left\{ \oint_{\gamma_\theta} f : \gamma_{\theta} \text{ is a closed geodesic} \right\}.
\]

For a sufficiently regular \( f \), the spectrum \( \mathcal{B}_{M}(f) \)  provides a rigidity for Anosov metrics, as shown in the following theorem:
\begin{teo}\label{Rig in M}\emph{\cite[Theorem 1.1]{Dairbekov}}
Let \( (M,g) \) be an Anosov manifold. If \( f \in C^{\infty}(M) \) satisfies \( \mathcal{B}_{M}(f) = \{0\} \), then \( f \) is identically zero.
\end{teo}
Using our methods, we obtain a significant generalization of Theorem \ref{Rig in M}:


\begin{proof}[\emph{\textbf{Proof of of Theorem \ref{integral_geodesic}}}]
Let \( f \in C^{\infty}(M) \). Then \( f \circ \pi \in C^{\infty}(SM) \), where \( \pi \colon SM \to M \) is the canonical projection. We observe that
\[
\mathcal{B}_{M}(f) = \mathcal{B}(f \circ \pi, \phi^t_{M}).
\]
\begin{itemize}
    \item[\text{(a)}] If $\mathcal{B}_{M}(f)$ is bounded, then \( \mathcal{B}(f \circ \pi, \phi^t_{M}) \) is bounded. By Corollary \ref{C3-GS}, it follows that \( \mathcal{B}(f \circ \pi, \phi^t_{M}) = \{0\} \). Hence \( \mathcal{B}_{M}(f) = \{0\} \), and the result follows from Theorem \ref{Rig in M}.
    \item[\text{(b)}] In this case, the flow version of Theorem \ref{Rigidity} implies that \( f \circ \pi \) is cohomologous to \( a_{i_0} \) for some index \( i_0 \). In particular, \( \mathcal{B}_{M}(f - a_{i_0}) = \{0\} \). Combining this with condition (a), we conclude that \( f \) must be constant and equal to \( a_{i_0} \).    
\end{itemize}

\end{proof}

\textbf{Acknowledgments:}
The author thanks Mingyang Xia for his invaluable insights, which were essential to the completion of this work.

\bibliographystyle{plain}
\pagestyle{empty}


\noindent Sergio Roma\~{n}a\\
School of Mathematics (Zhuhai), Sun Yat-sen University, 519802, China\\
sergio@mail.sysu.edu.cn

\end{document}